\documentclass[11pt,reqno]{amsproc}

\title{Topological Data Analysis and Cosheaves}

\author{Justin M.~Curry}
\address{Department of Mathematics,
        Duke University,
        Durham, NC USA}
\email{curry@math.duke.edu}

\usepackage{amsmath,mathrsfs}
\usepackage{amssymb,latexsym}
\usepackage{eulervm}
\usepackage{graphicx}
\usepackage[usenames,dvipsnames]{color}
\usepackage{xypic}
\usepackage[cmtip,arrow]{xy}
\usepackage[margin=0.8in]{geometry}

\usepackage{hyperref}
\hypersetup
{
    colorlinks, %
    citecolor=Blue,%
    filecolor=Blue,%
    linkcolor=Blue,%
    urlcolor=Gray
}

\theoremstyle{definition}
\newtheorem{thm}{Theorem}[section]
\newtheorem{lem}[thm]{Lemma}
\newtheorem{defn}[thm]{Definition}

\newtheorem{prop}[thm]{Proposition}
\newtheorem{ex}[thm]{Example}

\newtheorem{rmk}[thm]{Remark}

\renewcommand{\sfdefault}{iwona}
\DeclareMathAlphabet{\mathbfsf}{\encodingdefault}{\sfdefault}{bx}{n}
\newcommand{\Hom}{\mathrm{Hom}}
\newcommand{\obj}{\mathrm{obj}}

\newcommand{\Vect}{\mathbfsf{Vect}}

\newcommand{\Open}{\mathbfsf{Open}}

\newcommand{\Int}{\mathbfsf{Int}}

\newcommand{\Entr}{\mathbfsf{Entr}}

\newcommand{\cat}{\mathbfsf{C}}
\newcommand{\dat}{\mathbfsf{D}}

\newcommand{\Iat}{\mathbfsf{I}}

\newcommand{\hF}{\widehat{F}}

\newcommand{\hP}{\widehat{P}}

\newcommand{\id}{\mathrm{id}}

\newcommand{\cok}{\mathrm{cok}}
\newcommand{\im}{\mathrm{im}}

\newcommand{\rank}{\mathrm{rank}}

\newcommand{\RR}{\mathbb{R}}

\newcommand{\N}{{\mathbb{N}}}
\newcommand{\R}{{\mathbb{R}}}

\newcommand{\Z}{{\mathbb{Z}}}

\newcommand{\OO}{\mathcal{O}}

\newcommand{\cU}{{\mathcal U}}

\begin{document}

\maketitle

\begin{abstract}
This paper contains an expository account of persistent homology and its usefulness for topological data analysis. An alternative foundation for level set persistence is presented using sheaves and cosheaves.
\end{abstract}

\section{Introduction}
\label{sec:intro}

Topological data analysis (TDA) is a new area of research that uses algebraic topology to extract non-linear features from data sets. TDA has had marked success in identifying novel subtypes of breast cancer~\cite{nicolau2011topology,lum2013extracting}, extracting structure from the space of natural images~\cite{mumford-data-set}, determining coverage in sensor networks~\cite{de2007coverage}, and tackling many other problems in science and engineering.

In this paper we provide an expository introduction to one branch of TDA known as persistent homology, which was first introduced in~\cite{edelsbrunner2000topological}. We motivate homology and functoriality through examples, which we develop theoretically in the simplicial case. Barcodes are introduced as a convenient visual aid for picturing functoriality in persistence, as well as many other situations in mathematics. 

Outlining a foundation for level set persistence, which generalizes and includes sub-level set persistence as a special case, makes up the bulk of the second half of the paper. The simplicial Leray cosheaves are introduced as a first approximation to studying general level set persistence. To provide a canonical definition for level set persistence, a brief treatment of categories, functors and sheaves is presented. Finally, the entrance path category is introduced as an ideal indexing category for level set persistence that works in higher dimensions for definable maps.

\section{An Intuitive Introduction to Persistence}
\label{sec:intuitive-intro}

Traditionally, the scientific method informs data analysis in the following way: one creates a model, one runs an experiment to obtain data, and then one inspects whether or not the observed data fits the expected model. This method works beautifully in certain areas of science, most notably physics, where a great deal of theory has been developed and experiments continue to be conducted.

Today's problems of ``big data,'' where we have collected data without a particular hypothesis to test, shows that the process of discovery exhibited by physics cannot be reliably imitated. For example, in certain fields of cell biology, we can measure many quantities of interest, but inferring the underlying gene regulatory network is extremely challenging~\cite{Boczko2005}. Furthermore, there are many questions that are of interest to engineers and social scientists where deriving a causal model is not the goal, but rather one wants to automatically and rigorously extract features of interest from an already extant data set. In many situations the data in question often takes on interesting shapes that escape the reach of traditional methods~\cite{lum2013extracting}. 

Topological data analysis aims to provide additional tools for analyzing data sets that appear in science and engineering. These tools are not meant to replace existing techniques; rather, they provide an additional and powerful way for capturing intuitive (as well as not-so-intuitive) features in a data set. These methods focus on the ``shape'' of data and can be applied to data sets living in high dimensions.

\begin{figure}[ht]
  \centering
  \includegraphics[width=.5\textwidth]{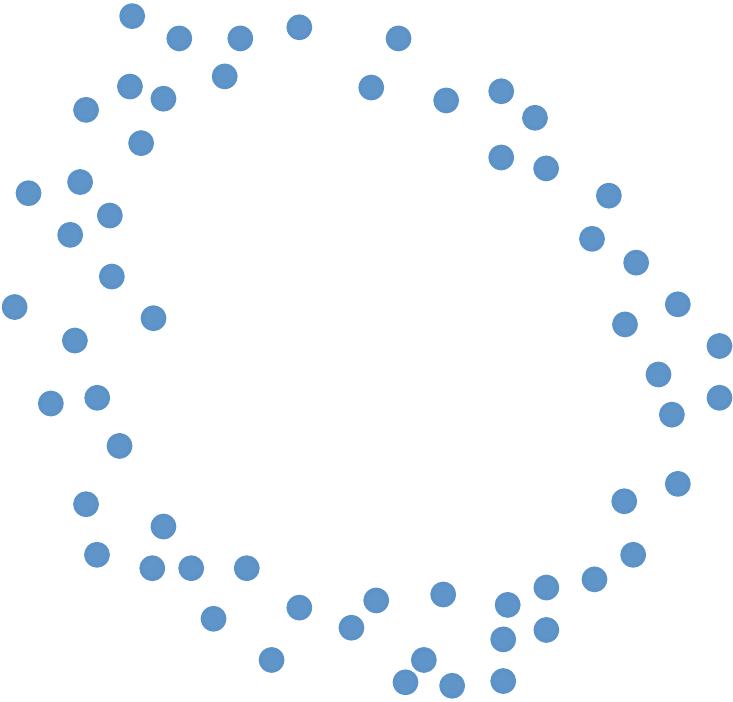}
  \caption{A point cloud $X$.}
  \label{fig:point_cloud_data}
\end{figure}

Consider a finite set of points in $\R^n$, which we call a \textbf{point cloud} for short. For example, our point cloud could be the set depicted in Figure \ref{fig:point_cloud_data}. One sees that the points appear to be sampled from a circle or an ellipse, but this observation is too informal. The first question we take up in this paper is ``How do we make this observation precise?'' If we are going to use descriptors such as ``looks like a circle'' for doing science, then we must use a new language that is precise, quantitative and computable.

\textbf{Homology} provides us with such a language. Homology is a mathematical theory of shape that is applicable to any suitably nice subset of $\R^n$ (as well as other, more general, types of spaces) that describes qualitative features that are invariant under continuous deformation. Such features include the number of connected pieces that a space $X$ breaks up into. Here we say a subset $X$ of $\R^n$ is connected if there is a continuous path in $X$ connecting any two points in $X$; said differently, in a connected space one is able to deform any one point to any other. Another feature that homology measures about a space $X$ is whether a loop in $X$ can be deformed to a single point in $X$ in such a way that the deformation never leaves $X$. Similar, higher-dimensional, features are also detected by homology, e.g.~whether a sphere is deformable to a point can be measured.

Each of the above examples of what homology measures is graded by dimension: points are 0-dimensional, loops are 1-dimensional, spheres are 2-dimensional, and so on. This is because homology is similarly graded by dimension. Homology defines, for each non-negative integer $i$, topological space $X$, and abelian group $G$, a new group
\[
H_i(X;G)
\] 
called the \textbf{$i^{th}$ homology group of $X$}. 

\begin{rmk}\label{rmk:group}
We will assume that $G$ is a field $k$ (such as the reals $\R$ or the field with two elements $\mathbb{F}_2$) so that each homology group is actually a vector space, which we will write as $H_i(X)$. We will continue to use the term ``group'' out of convention, even though ``vector space'' is meant.
\end{rmk}

The elements of the homology group $H_i(X)$ are equivalence classes of certain $i$-dimensional features. For example, two loops that are deformable to each other represent the same element of $H_1(X)$. A full treatment of homology is beyond the scope of this paper, but there are many thorough textbooks on homology, such as~\cite{bredon1993topology,hatcher,spanier}, and a precise version of suitable applicability is developed in Section~\ref{sec:simplicial-complexes}, following~\cite{munkres1984elements}.

Foregoing this more precise treatment of homology, let us describe the homology groups for various subsets of $\R^2$. For example, the homology groups of the subset 
$$C:=\{(x,y)\,|\, x^2+y^2=1\}\subset\R^2,$$ 
which is a more traditional definition of a circle, are
\[
H_0(C)=k \qquad H_1(C)=k \qquad H_i(C)=0 \quad i\geq 2.
\]
If we were to move or stretch the subset $C$, we'd get the same result. If we viewed the circle as lying inside the first two coordinates of the space $\R^{10}$, we'd get the same result. Homology is an intrinsic invariant of a space, with no regard to its embedding in another space.

Let us now view the set of points in Figure \ref{fig:point_cloud_data} using the lens of homology. Foregoing explicit computation, we observe that this picture has the homology groups
 \[
 H_0(X)=k^{60} \qquad H_i(X)=0 \quad i\geq 1,
 \]
which corresponds to the 60 points in the data set and the lack of circles or other homological features. At this point, homology does not confirm our intuition that the data looks like a circle. To remedy this, let us fatten each point in $X$ by including the points that are within distance $r$ of some point of $X$. If we denote the closed ball by $B(x,r)=\{y\in\R^n\,|\,||x-y||\leq r\}$, then our fattened space will be denoted $X_r:=\cup_{x_i\in X} B(x_i,r)$. In Figure~\ref{fig:point_cloud_3sizes} we have depicted these fattened spaces for three different radii.
\begin{figure}[ht]
  \centering
  \includegraphics[width=\textwidth]{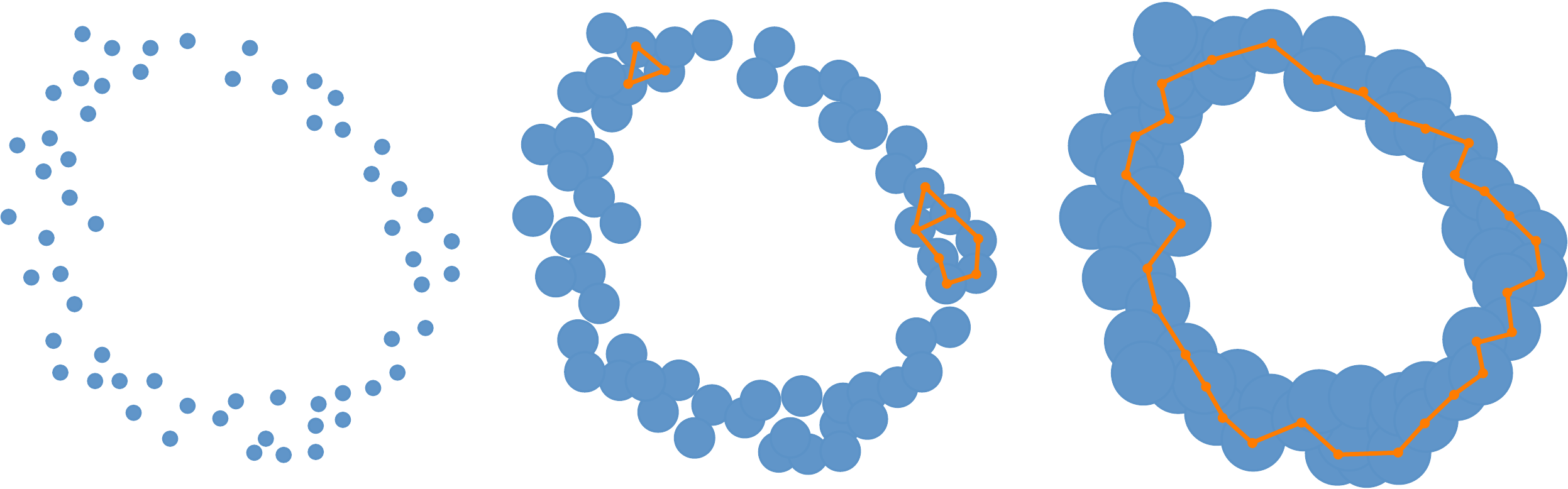}
  \caption{An ``augmented'' point cloud $X_r=\cup B(x_i,r)$ for three different radii $r_0<r_1<r_2$.}
  \label{fig:point_cloud_3sizes}
\end{figure}

The first radius $r_0$ is chosen so that $X_{r_0}$ has the same homology as $X$. The radii $r_1$ and $r_2$ are chosen so that the spaces $X_{r_1}$ and $X_{r_2}$ have homology groups different from $X_{r_0}$. For $X_{r_1}$ we have, again without calculation,
\[
H_0(X_{r_1})=k^{11} \qquad H_1(X_{r_1})=k^3 \qquad H_i(X_{r_1})=0 \quad i\geq 2,
\]
which corresponds to the eleven connected pieces (it may be hard to resolve in Figure~\ref{fig:point_cloud_3sizes} whether certain balls are touching or not and this will affect the value of $H_0(X_{r_1})$), the three small holes we have outlined with edges in a graph, and no higher features such as caves. Finally, when one considers a large enough radius $r_2$, we get a homology computation of
\[
H_0(X_{r_2})=k \qquad H_1(X_{r_2})=k \qquad H_i(X_{r_2})=0 \quad i\geq 2,
\]
which is exactly the answer we provided for the circle $C$. We have captured the apparent circle in Figure~\ref{fig:point_cloud_data} by using homology and this fattening procedure. This is the premier example of persistent homology's effectiveness for capturing shape in a point cloud.

\begin{rmk}
In fact, this procedure captures even more. One can estimate the radius of the circle in $X_{r_2}$ by determining the first radius $r_3>r_2$ where the homology group $H_1(X_{r_3})=0$. This is surprising because homology is invariant under bending and stretching and the radius of a circle is not. The difference is that we are considering a \emph{family} of homology groups over the half real line $\{r\geq 0\}\subset \R$ and the length (a geometric property) over which the homology group $H_1(X_r)=k$ (an algebraic property) gives us an estimate for perceived radius of the point cloud $X$. For a fascinating application of this idea to fractals and self-similar shapes that appear in physics see~\cite{macpherson-shape-topology}.
\end{rmk}

\section{Simplicial Complexes, Homology and Functoriality}
\label{sec:simplicial-complexes}

Now that the reader has some intuition for homology in low degrees and its practicality for data analysis, we introduce a simpler variant of homology defined for simplicial complexes, which are combinatorial models for topological spaces.

\subsection{Simplicial Complexes}

\begin{defn}[Simplicial Complex]\label{defn:simplicial-cplx}
Given a set $V$, a \textbf{simplicial complex} $K$ is a collection of subsets of $V$, such that if $\tau\in K$, then any subset of $\tau$ is also in $K$. Said differently, a simplicial complex $K$ is a subset of the power set $P(V)$ such that
\[
\mathrm{if}\,\,\tau\in K\,\,\mathrm{and}\, \sigma\subseteq \tau \,\, \mathrm{then}\,\, \sigma\in K.
\]
One calls the elements of $K$ \textbf{simplices}. If the cardinality of $\sigma$ is $n+1$, one says that $\sigma$ is an \textbf{$n$-simplex}.
\end{defn}

\begin{ex}\label{ex:interval}
Suppose $V$ is a set with two elements $x$ and $y$. The maximal simplicial complex on $V$ is the set of all non-empty subsets, i.e.~$K=\{\{x\},\{y\},\{x,y\}\}$. The subsets $\{x\}$ and $\{y\}$ are the $0$-simplices of $K$ and $\{x,y\}$ is the only $1$-simplex. This simplicial complex is usually thought of as an undirected graph with two vertices and one edge.
\end{ex}

Below are some more interesting sources of simplicial complexes, which describe the shapes previously considered using finite, simplicial complexes.

\begin{ex}[\v{C}ech Complex]\label{ex:cech-nerve}
Suppose $X$ is a point cloud. For each radius $r>0$ we can construct the \textbf{\v{C}ech complex} $\check{C}_r(X)$ using the set of points in $X$ for a vertex set. A collection of points $\sigma=\{x_{i_0},\ldots,x_{i_n}\}\subseteq X$ defines an $n$-simplex in $\check{C}_r(X)$ if and only if the intersection of closed balls of radius $r$ is nonempty, i.e.~$\cap_{j=0}^nB(x_{i_j},r)\neq \varnothing$.
\end{ex}

\begin{ex}[Vietoris-Rips Complex]
Suppose again that $X$ is a point cloud. We can build a simplicial complex on $X$ using another construction called the \textbf{Vietoris-Rips complex}, or ``Rips complex'' for short, $V_r(X)$ by declaring a list of vertices $\sigma=\{x_{i_0},\ldots,x_{i_n}\}\subseteq X$ to be a simplex if the maximum distance between any two points in $\sigma$ is at most $2r>0$.
\end{ex}

\begin{rmk}
For purposes of computation, the Rips complex is preferred over the \v{C}ech complex. Determining whether a collection of points defines a simplex in the Rips complex can be done simply by computing pairwise distances between points in $X$. However, determining whether a collection of points defines a simplex in the \v{C}ech complex requires determining whether there is some unknown point in the ambient space $\R^n$ that is at most distance $r$ away from the collection. 

Fortunately, there is a comparison theorem that relates the two constructions for a point cloud in $\R^n$. Although every simplex in the \v{C}ech complex at radius $r$ defines a simplex in the Rips complex at radius $r$, the converse is not true, as the reader can check for three points forming an equilateral triangle in the plane. In~\cite{de2007coverage} the authors prove that every simplex in $V_r(X)$ is a simplex in $\check{C}_{\sqrt{2}r}(X)$. These two observations are expressed by the sequence of inclusions
\[
  V_r(X)\subseteq \check{C}_{\sqrt{2}r}(X)\subseteq V_{\sqrt{2}r}(X).
\]
\end{rmk}

\begin{ex}[Nerve]\label{ex:nerve}
Let $\cU:=\{U_i\}_{i\in I}$ be a collection of subsets of a space $X$. The indexing set $I$ can serve as the vertex set for a simplicial complex called the \textbf{nerve}, which we denote by $N_{\cU}$. A set of indices $\sigma:=\{i_0,\ldots,i_n\}\subseteq I$ defines a simplex if and only if the corresponding intersection of sets $U_{\sigma}:=U_{i_0}\cap\cdots\cap U_{i_n}\neq \varnothing$. One can easily see that any subset $\gamma\subset\sigma$ is also a simplex, so that this rule does indeed define a simplicial complex.
\end{ex}

In the next section we will rigorously define the homology of a simplicial complex, called simplicial homology. This definition uses simplices in a very explicit way, but it should be noted that there are other notions of homology, e.g.~singular homology, that only requires the structure of a topological space, such as a subset of $\R^n$. It was an important question as to whether or not singular homology of the space $X_r=\cup B(x_i,r)$ is the same as simplicial homology of the \v{C}ech complex $\check{C}_r(X)$. The answer is yes, and involves some very technical results that have been developed over the past 100 years: the homotopy invariance of singular homology, the equivalence of singular homology and simplicial homology, and the Nerve Theorem (whence the above construction came), all of which are covered in detail in~\cite{hatcher}.

\subsection{Homology for Simplicial Complexes}

Suppose $K$ is a simplicial complex equipped with a total ordering of the vertex set $V$ so that one can speak meaningfully of comparisons such as $v_{i_0}<v_{i_1}<\cdots$ and so on. We use this order to present any simplex $\sigma$ in $K$ as an ordered list of vertices $\sigma=[v_{i_0},\ldots,v_{i_p}]$.

\begin{defn}
The \textbf{boundary of a simplex} $\sigma$, written $\partial \sigma$, is the following formal linear combination
\[
  \partial\sigma=[v_{i_1},\ldots,v_{i_p}]-[v_{i_0},v_{i_2},\ldots]+\cdots+(-1)^p[v_{i_0},\ldots,v_{i_{p-1}}].
\]
\end{defn}

\begin{ex}\label{ex:interval-bdy}
Suppose $K$ is the simplicial complex described in Example~\ref{ex:interval}. The two 0-simplices have empty boundary, so we stipulate that $\partial [x]=\partial [y]=0$. Choosing the order $x<y$, we denote the unique oriented 1-simplex in $K$ by $a=[x,y]$. One can check that
\[
\partial a=\partial [x,y]=[y]-[x].
\]
\end{ex}

\begin{defn}
Given a simplicial complex $K$, define the \textbf{group of $p$-chains} $C_p(K)$ as the vector space spanned by all simplices in $K$ of cardinality $p+1$. Every basis vector can be referred to by the ordered presentation of its vertices, e.g.~$\sigma=[v_{i_0},\ldots,v_{i_p}]$. The \textbf{boundary operator} $\partial_p:C_{p}(X)\to C_{p-1}(X)$ is the linear map gotten by extending the definition of the boundary of a simplex linearly, i.e.~$\partial_p(\sigma_1+\sigma_2)=\partial\sigma_1+\partial\sigma_2$. 
\end{defn}

The most important property of the boundary operator is that $\partial_{p}\circ\partial_{p+1}=0$ for every integer $p\geq 0$, which the reader can check for themselves or find as Lemma 5.3 of~\cite{munkres1984elements}. This system of identities is often summarized simply as $\partial^2=0$, the upshot of which is that $\im \partial_{p+1}\subseteq \ker \partial_p$. This observation is essential for the definition of homology.

\begin{defn}
The \textbf{$p^{th}$ simplicial homology group of $K$} is defined to be the quotient $k$-vector space
\[
  H_p(K)=\frac{\ker \partial_p}{\im \partial_{p+1}}
\]
Elements of $\ker\partial_p$ are called \textbf{cycles} and elements of $\im\partial_{p+1}$ are called \textbf{boundaries}. Any cycle in $C_p(K)$ that is the boundary of a cycle in $C_{p+1}(K)$ is regarded as zero in $H_p(K)$ and any cycle that is not a boundary specifies a non-zero element of $H_p(K)$.
\end{defn} 

\begin{ex}\label{ex:interval-homology}
We can now compute the homology groups of the simplicial complex described in Example~\ref{ex:interval}, by using the boundary calculation in Example~\ref{ex:interval-bdy} and the above definition. Since there are no $p$-simplices for $p>1$, we have that $C_p(K)=0$ for $p>1$. The vector space $C_1(K)$ is one-dimensional, generated by the simplex $a$. It's boundary is $[y]-[x]$, which is not zero, so $\ker\partial_1=0$ and thus $H_1(K)=0$. Since $C_0(K)$ is two-dimensional, generated by $[x]$ and $[y]$, and the image of $\partial_1$ is one-dimensional, spanned by $[y]-[x]$, we can conclude that $H_0(K)$ is one-dimensional. To summarize
\[
H_p(K)=0 \qquad p>0 \qquad \mathrm{and} \qquad H_0(K)=k.
\] 
This reflects the fact that the simplicial complex $K$ is connected and has no other homological features.
\end{ex}

\begin{rmk}[Cohomology]
Homology has a mirror image called \textbf{cohomology}. In place of the group of $p$-chains $C_p(K)$ one studies the vector space of linear functionals on the $p$-simplices of $K$. We define the \textbf{group of cochains} $C^p(K):=C^{*}_p(K)$ to be the set of linear maps from $C_p(K)$ to the field $k$. Since the map $\partial_{p+1}$ maps $C_{p+1}(K)$ to $C_p(K)$, any functional on $C_p(K)$ becomes a functional on $C_{p+1}(K)$ by applying $\partial_{p+1}$ first. This is the standard construction of the transpose $(\partial_{p+1})^{T}$, which we call the \textbf{coboundary operator} and write as $\delta^{p}$. One can easily check that the condition $\partial_{p}\circ\partial_{p+1}=0$ implies $\delta^{p+1}\circ\delta^p=0$, thus allowing us to define the \textbf{$p^{th}$ cohomology group} as
\[
  H^p(X)=\frac{\ker \delta^p}{\im \delta^{p-1}}.
\]
For technical reasons, cohomology is a better invariant than homology, but when $K$ is a finite simplicial complex the vector spaces $H_p(K)$ and $H^p(K)$ are isomorphic.
\end{rmk}

\subsection{The Necessity of Functoriality}
\label{subsec:functoriality-by-example}

\begin{figure}[ht]
    \centering
    \includegraphics[width=.9\textwidth]{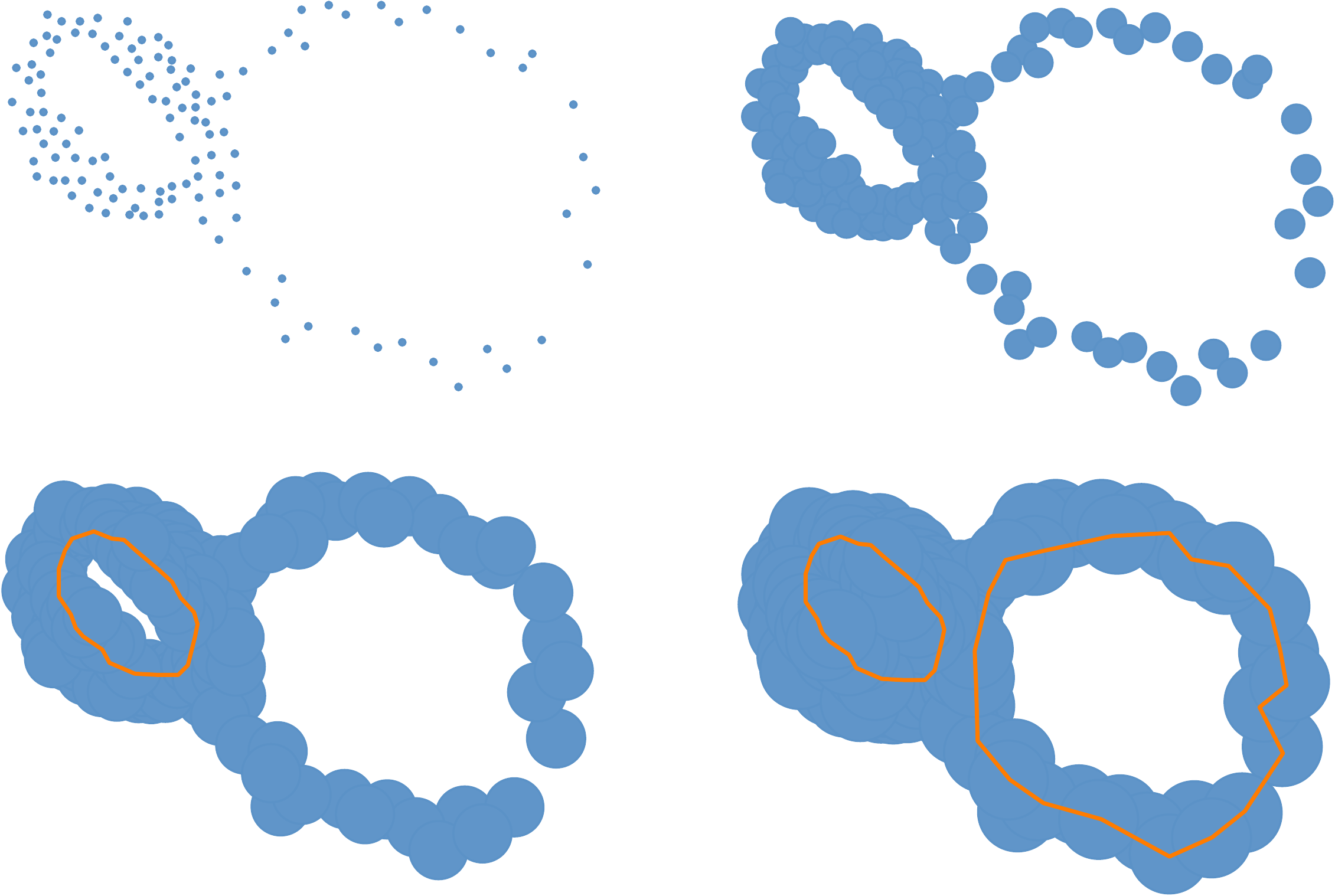}
    \caption{An augmented point cloud $X_{r_i}$ at four different radii $r_0 < r_1 < r_2 < r_3$.}
    \label{fig:pointcloud_functor}
  \end{figure} 

Recall that we are trying to understand the shape of a point cloud $X$ via the homology of the augmented spaces $X_r=\cup B(x_i,r)$. We do this first by computing the homology of the \v{C}ech complex $\check{C}_r(X)$ or, if one is willing to trade accuracy for efficiency, by computing the homology of the Vietoris-Rips complex $V_r(X)$ for varying values of $r>0$. One might try to summarize the homology groups $H_i(X_r)$ for varying $r$ by graphing the dimension of $H_i(X_r)$ as a function of $r$, but this turns out to be misleading; one can mistake a point-cloud with two circles for just one, as Figure~\ref{fig:pointcloud_functor} illustrates. 

In Figure~\ref{fig:pointcloud_functor}, the radius $r$ required to form the big circle on the right is exactly large enough to cause the smaller left circle to disappear. If one wants to discriminate the point clouds presented in Figure~\ref{fig:point_cloud_data} and the upper left hand corner of Figure~\ref{fig:pointcloud_functor}, then one needs more than the dimension of the homology groups for varying radii $r$; instead, one needs to utilize the \textbf{functoriality} of homology. 
\begin{defn}
To say homology is \textbf{functorial} is to say the following: to each continuous map $f:X\to Y$ and integer $i\geq 0$ homology associates a linear map $f_*:H_i(X)\to H_i(Y)$. Intuitively-speaking, this means that a map of spaces defines a map between the corresponding homological features.
\end{defn}

In the bottom row of Figure~\ref{fig:pointcloud_functor} we have a space $X_{r_2}=\cup B(x_i,r_2)$ that includes into $X_{r_3}=\cup B(x_i,r_3)$. This is clear from the definition: if $r_2 < r_3$, then $B(x_i,r_2)\subset B(x_i,r_3)$ and thus there is an inclusion $\iota_{3,2}:X_{r_2}\hookrightarrow X_{r_3}$. A simple calculation reveals that the induced map on first homology is the zero map, i.e.
\[
(\iota_{3,2})_*:H_1(X_{r_2})\to H_1(X_{r_3}) \qquad \mathrm{is} \qquad 0: k\to k.
\]
This calculation captures the observation that the circle on the left is unrelated to the circle on the right. Specifically, the image of the circle in $X_{r_2}$ under the inclusion yields a circle that is the boundary of a disc in $X_{r_3}$ and thus zero in the vector space $H_1(X_{r_3})$.

To contrast this example with what happens in our first example depicted in Figure~\ref{fig:point_cloud_3sizes}, we can observe that once the one large generator for $H_1(X_{r}) $ appears, it is mapped isomorphically onto generators for $H_1(X_{s})$ for $r_{\mathrm{min}}<r<s<r_{\mathrm{max}}$, where $r_{\mathrm{min}}$ refers to the minimum radius required for the ``small'' holes to disappear (as pictured in the middle of Figure~\ref{fig:point_cloud_3sizes}) and $r_{\mathrm{max}}$ corresponds roughly to the radius of the annulus pictured to the right in Figure~\ref{fig:point_cloud_3sizes}.

\subsection{Functoriality for Simplicial Maps}

Although singular homology is functorial for arbitrary continuous maps, a precise version of functoriality for simplicial maps communicates the essential details of how the maps $f_*:H_i(X)\to H_i(Y)$ are defined.

\begin{defn}
Suppose $K$ and $L$ are simplicial complexes. A \textbf{simplicial map} is a map from the vertex set of $K$ to the vertex set of $L$ with the property that if $\sigma$ is a simplex of $K$, then $f(\sigma)$ is a simplex in $L$.
\end{defn}

One of the important properties of a simplicial map is that it takes $p$-simplices of $K$ to $m$-simplices of $L$ as long as $m\leq p$. This implies that there is a map of vector spaces
\[
  C_p(f):C_p(K)\to C_p(L)
\]
where if the image of a $p$-simplex is of dimension less than $p$, then we declare $C_p(f)$ of that simplex to be zero.

If we consider the maps $C_p(f)$ for various $p$ at once, we see that we have a ladder of maps
\[
\xymatrix{ \cdots \ar[r] & C_p(K) \ar[r]^{\partial^K_p} \ar[d]_{C_p(f)} & C_{p-1}(K) \ar[r] \ar[d]^{C_{p-1}(f)} & \cdots \\
\cdots\ar[r] & C_p(L) \ar[r]_{\partial^L_p} & C_{p-1}(L) \ar[r] & \cdots}
\]
with the additional property that
\[
C_{p-1}(f)\circ \partial_p^K = \partial_p^L \circ C_{p}(f) \qquad \forall p\geq 0.
\]
Such a collection of maps is called a \textbf{chain map} and has the property that it induces a well-defined map on homology.

\begin{lem}[Lemma 12.1 of~\cite{munkres1984elements}]
Given a simplicial map $f:K\to L$, the chain map $C_{\bullet}(f):C_{\bullet}(K)\to C_{\bullet}(L)$ induces well-defined maps between homology groups.
\[
f_*:H_i(K)\to H_i(L)
\]
\end{lem}

\section{Barcodes: Visualizations of Functoriality}
\label{sec:barcodes}

As described at the beginning of Section~\ref{subsec:functoriality-by-example}, one must use the homology groups of the $X_r$ as well as the induced maps on homology $H_i(X_r)\to H_i(X_{r'})$ for $r<r'$ in order to capture homological features that persist over varying radii. This information is collectively called a \textbf{persistence module} and is defined below.

Despite the complexity inherent to persistence modules, there are two methods for visualizing persistence modules that have had success in making TDA easier to understand by non-mathematicians. The first method of visualizing persistence is the \textbf{persistence diagram}, which we describe in Remark~\ref{rmk:pers-diag}. The persistence diagram came first and was developed simultaneously with persistent homology~\cite{cohen2007stability,edelsbrunner2000topological} and is still widely used today~\cite{bendich2014persistent}. The second method of visualization is the \textbf{barcode} and it was developed by Carlsson, Zomorodian, Collins and Guibas~\cite{carlsson2004persistence} after they reformulated the definition of persistent homology provided by Edelsbrunner, Letscher and Zomorodian~\cite{edelsbrunner2000topological}. In this section we describe the barcode construction from a modern perspective using a recent theorem of Crawley-Boevey~\cite{crawley2012decomposition}. We prefer the barcode method only because it is more useful for visualizing results from the (co)sheaf-theoretic perspective developed later in the paper.

\begin{defn}\label{defn:persistence_module}
Let $(\R,\leq)$ denote the reals with its total ordering. A \textbf{persistence module} consists of a collection of vector spaces $\{V_t\}_{t\in\R}$, one for each real number $t$, and a collection of linear maps $\rho_{t,s}:V_s\to V_t$ for every pair of numbers $s\leq t$. Moreover, we require that if one has a triple $r\leq s\leq t$, then $\rho_{t,r}=\rho_{t,s}\circ\rho_{s,r}$. We denote a persistence module by $(V,\rho^V)$, but we may suppress the $V$ in $\rho^V$ or even drop the $\rho^V$ altogether.
\end{defn}

\begin{rmk}
Definition~\ref{defn:persistence_module} works equally well for any totally ordered set, just as it was defined in~\cite{carlsson2004persistence}. Consequently, we will sometimes shift from the reals $(\R,\leq)$ to the integers $(\Z,\leq)$ or to the natural numbers $(\N,\leq)$ and still use the terminology of persistence modules.
\end{rmk}

Observe that one can add two persistence modules to create a third persistence module, i.e. if $(V,\rho^V)$ and $(W,\rho^W)$ are two persistence modules, then one obtains a third persistence module $(U,\rho^U)$ by defining $U_t:=V_t\oplus W_t$ and $\rho_{t,s}^U:=\rho^V_{t,s}\oplus \rho^W_{t,s}$. We denote the sum by $(V\oplus W,\rho^V\oplus\rho^W)$ or more simply by $V\oplus W$.

There is a fundamental structure theorem for persistence modules, due to Crawley-Boevey~\cite{crawley2012decomposition}, that explains how any persistence module can be written as a direct sum of simpler persistence modules. We now describe these simpler persistence modules.

\begin{defn}\label{defn:interval}
An \textbf{interval} in $(\R,\leq)$ is a subset $I\subset \R$ having the property that if $r,t\in I$ and if there is an $s\in \R$ such that $r\leq s\leq t$, then $s\in I$ as well. An \textbf{interval module} $k_I$ assigns to each element $s\in I$ the vector space $k$ and assigns the zero vector space to elements in $\R\setminus I$. All maps $\rho_{t,s}$ are the zero map, unless $s,t\in I$ and $s\leq t$, in which case $\rho_{t,s}$ is the identity map. 
\end{defn}

Since interval modules are completely determined by the interval where they assign non-zero vector spaces, we can draw a \textbf{bar} to represent an interval module. The following structure theorem shows that any persistence module can be represented by a collection of bars, called a \textbf{barcode}.

\begin{thm}[Decomposition for Pointwise-Finite Persistence Modules~\cite{crawley2012decomposition}]\label{thm:crawleyboevey}

If $(V,\rho^V)$ is a persistence module for which every vector space $V_t$ is finite-dimensional, then the module is isomorphic to a direct sum of interval modules, i.e.
\[
V\cong \bigoplus_{I\in D} k_I.
\]
Here $D$ is a multi-set of intervals. A multi-set is a set allowing repetitions, i.e. a set equipped with a function $\mu$ indicating the multiplicity of each given element.

\end{thm}

\begin{rmk}
It should be noted that the definition of a barcode first appears in 2004~\cite{carlsson2004persistence}, but the above theorem, which is used to prove that every persistence module has a presentation as a barcode, was only proved in 2012~\cite{crawley2012decomposition}. The reason is that~\cite{carlsson2004persistence} uses a standard classification theorem for finitely generated modules over a principal ideal domain described in~\cite{zomorodian2005computing}, which only works when the indexing set is $(\Z,\leq)$ rather than $(\R,\leq)$.
\end{rmk}

\begin{rmk}\label{rem:undirected-bars}
When the indexing set is $(\Z,\leq)$ the conclusion of Theorem~\ref{thm:crawleyboevey} does not actually depend on the direction of the arrows in the persistence module. This means that when we considered \textbf{zig-zag modules}, i.e.~vector spaces and maps of the form 
\[
  \cdots V_n \leftarrow V_{n+1} \rightarrow V_n \leftarrow V_{n+2} \cdots
\]
with integer indexing, they will have a decomposition into bars as well.
\end{rmk}

\subsection{Barcodes in Linear Algebra}
For this section, let us assume that all of our persistence modules are indexed by the integers $(\Z,\leq)$. In this setting, Crawley-Boevey's theorem, which is a generalization of much older results in quiver representation theory~\cite{derksen}, summarizes a great deal of elementary linear algebra. For example, it has the fundamental theorem of linear algebra as a consequence~\cite{strang1993fundamental}, i.e. any map of vector spaces $T:V\to W$ has a matrix representation that is diagonal with $0$ and $1$ entries, the number of 1s corresponding to the rank of the matrix, cf.~\cite{artin1991algebra} Chapter 4, Proposition 2.9. Said differently, there are vector space isomorphisms making the following diagram commute:

\[
    \xymatrix{ V \ar[r]^T \ar[d]_{\varphi}^{\cong} & W \ar[d]^{\psi}_{\cong} \\
    \im(T) \oplus \ker(T) \ar[r]_{\id\oplus 0} & \im(T)\oplus \cok(T)}
\]
Here $\im(T)$, $\ker(T)$, and $\cok(T)$ refer to the image, kernel and cokernel of $T$ respectively. Although the image of $T$ is properly a subspace of $W$, the first isomorphism theorem identifies it with $V$ modulo the kernel.

\begin{ex}[Barcodes for Visualizing Rank]
Consider any linear map $T:\R^3\to\R^2$ as a persistence module by extending by zero vector spaces and maps. There are three isomorphism classes of such persistence modules determined by the rank of $T$. The associated barcodes are depicted in Figure~\ref{fig:rank012bars}.
\end{ex}

\begin{figure}
    \centering
    \includegraphics[width=.8\textwidth]{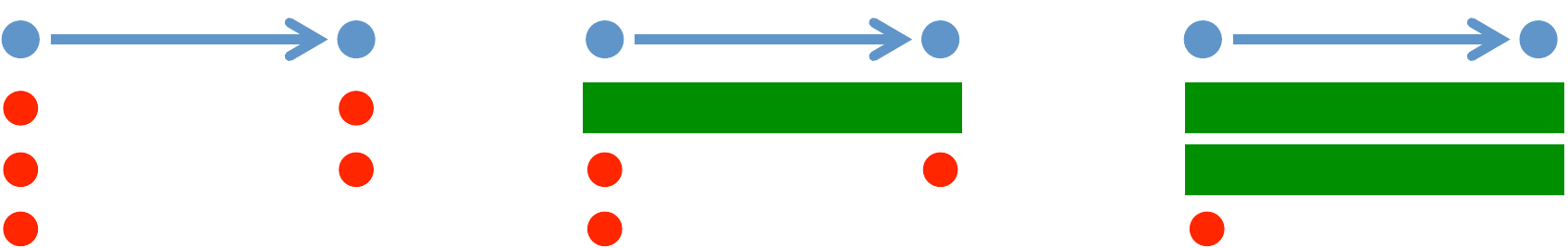}
    \caption{Barcodes associated to $T:\R^3\to\R^2$ for $\rank(T)=0,1,2$}
    \label{fig:rank012bars}
\end{figure}

\begin{ex}[Barcodes for Chain Complexes]\label{ex:chaincomplex_bc}
A chain complex of vector spaces is a special example of a persistence module where $\rho_{i+1}\circ\rho_i=0$. Consequently, every chain complex has a presentation as a barcode . With a moment's reflection on Theorem~\ref{thm:crawleyboevey} one can see that any chain complex can be written as the direct sum of two types of modules: 
the length zero interval modules
\[
S_i: \qquad \cdots \to 0 \to k \to 0 \to \cdots
\]
and the length one interval modules.
\[
P_i: \qquad \cdots \to 0 \to k \to k \to 0 \to \cdots
\]
Figure~\ref{fig:chaincomplex_bcs} gives a visual depiction of such a barcode decomposition. One should note that the process of taking homology of a chain complex corresponds precisely to deleting the green bars and leaving behind the red dots.
\end{ex}

\begin{figure}[ht]
    \centering
    \includegraphics[width=.8\textwidth]{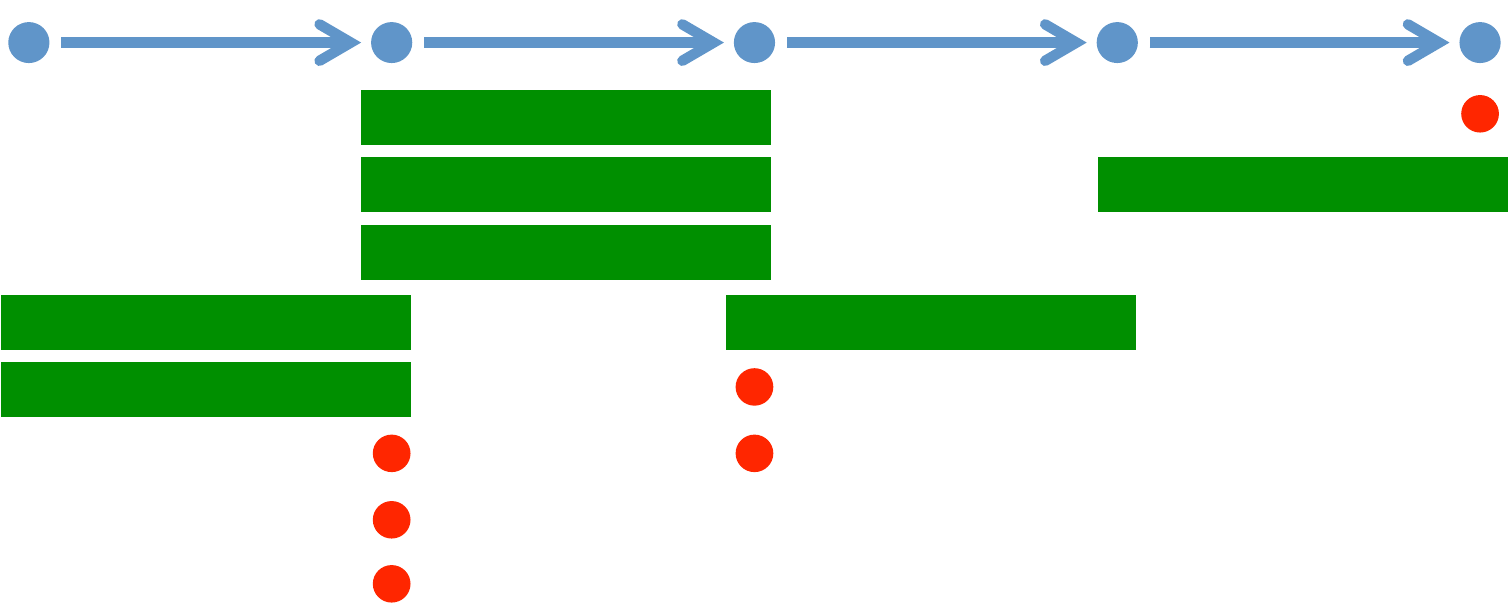}
    \caption{Barcode for a Chain Complex}
    \label{fig:chaincomplex_bcs}
\end{figure}

\begin{rmk}
If the reader is familiar with the notion of chain homotopy, one can observe that the green bars give a visualization of a chain homotopy between chain complexes: the first being the original chain complex and the second being the graded homology, viewed as a chain complex with zero maps between the homology groups. Thus Figure~\ref{fig:chaincomplex_bcs} provides a proof-by-picture of a standard exercise in homological algebra: that the derived category of chain complexes over a field is equivalent to the graded category of vector spaces~\cite{weibel}.
\end{rmk}

\subsection{Barcodes for Persistence}
Now we will return to persistence modules that are indexed by $(\R,\leq)$. As before, let $X$ be a point cloud. One can easily observe that as subsets of $\R^n$ we have the sequence of inclusions
\[
  X_{r_0}\hookrightarrow X_{r_1} \hookrightarrow X_{r_2} \hookrightarrow X_{r_3} \cdots
\]
whenever $r_0\leq r_1\leq r_2 \leq \cdots$ and so on. Taking the $i$th homology of this sequence of spaces and maps provides a persistence module:
\begin{equation*}
  H_i(X_{r_0}) \to H_i(X_{r_1}) \to H_i(X_{r_2}) \to H_i(X_{r_3}) \to \cdots 
\end{equation*}
By applying Theorem~\ref{thm:crawleyboevey}, we can determine the barcode of the point cloud $X$. Long bars (intervals that span a long range of radii) are considered to be robust topological signals in the data set. For Figure \ref{fig:point_cloud_data}, there would be one long bar in the persistence module corresponding to $H_0$, indicating that after a certain radius the space $X_r$ is connected, and another long bar in the module corresponding to $H_1$, indicating the apparent circle in the data set. To summarize, we have the following prototypical pipeline of topological data analysis.

\begin{defn}[Point Cloud Persistence]\label{defn:point-cloud-persistence}
The \textbf{point cloud persistence pipeline} consists of the following ingredients and operations:
\begin{enumerate}
\item Let $X$ denote a point cloud, i.e. the union of a finite set of points $\{x_i\}\subset \R^n$.
\item The union of balls $X_r:=\cup_{x_i\in X} B(x,r)$ and their inclusions (or alternatively the \v{C}ech or Rips complex and the inclusions of simplicial complexes)
defines for each $i\geq 0$ a persistence module:
\begin{equation*}
  H_{i}(X_{r_0}) \to H_{i}(X_{r_1}) \to H_{i}(X_{r_2}) \to H_{i}(X_{r_3}) \to \cdots 
\end{equation*}
\item Applying Theorem \ref{thm:crawleyboevey} provides a multiset of intervals, which is visualized as a barcode or a persistence diagram by the end user. 
\end{enumerate}
\end{defn}

\begin{rmk}[Persistence Diagrams]\label{rmk:pers-diag}
One can represent any interval $I\subset\R$ using its left-hand endpoint, which we call its \textbf{birth} $b(I)$, and its right-hand endpoint, which we call its \textbf{death} $d(I)$. We can then represent this as a point in the plane $\R^2$ via its coordinate pair $(b(I),d(I))$, where clearly $b(I)\leq d(I)$. In this way we can use Theorem~\ref{thm:crawleyboevey} to produce a multi-set of points in the plane from any persistence module. This multi-set of points is the persistence diagram.
\end{rmk}

\subsection{Barcodes from Sub-Level Sets}

The first and second steps of the persistence pipeline offer opportunities for endless modification and application. Instead of considering a point cloud, one can start with a space $X$ equipped with a function $f:X\to\R$ and consider the family of sub-level sets $X_r:=f^{-1}(-\infty,r]$. As long as the function and space are sufficiently nice, we can use Theorem \ref{thm:crawleyboevey} to produce a barcode.

In particular, this view generalizes the previous description in the following simple way. Given a point-cloud $X$ in $\R^n$, consider the function that for each point $p\in \R^n$ returns the minimum Euclidean distance from $p$ to some point in $X$, i.e.
\[
f(p)=\min_{x_i\in X} \{||p-x_i||\}.
\]
Clearly the sequence of augmented point clouds
\[
X_{r_0}\hookrightarrow X_{r_1} \hookrightarrow X_{r_2} \hookrightarrow \cdots
\]
is equal to
\[
f^{-1}(-\infty,r_0] \hookrightarrow f^{-1}(-\infty,r_1] \hookrightarrow f^{-1}(-\infty,r_2] \hookrightarrow \cdots
\]


When the space $X$ has the structure of a manifold and $f:X\to\R$ is differentiable, sub-level set persistence provides a new perspective on Morse theory, which describes precisely how the homology of the sub-level set $X_t$ changes when $t$ passes through a critical value of $f$.

\begin{figure}[ht]
    \centering
    \includegraphics[width=\textwidth]{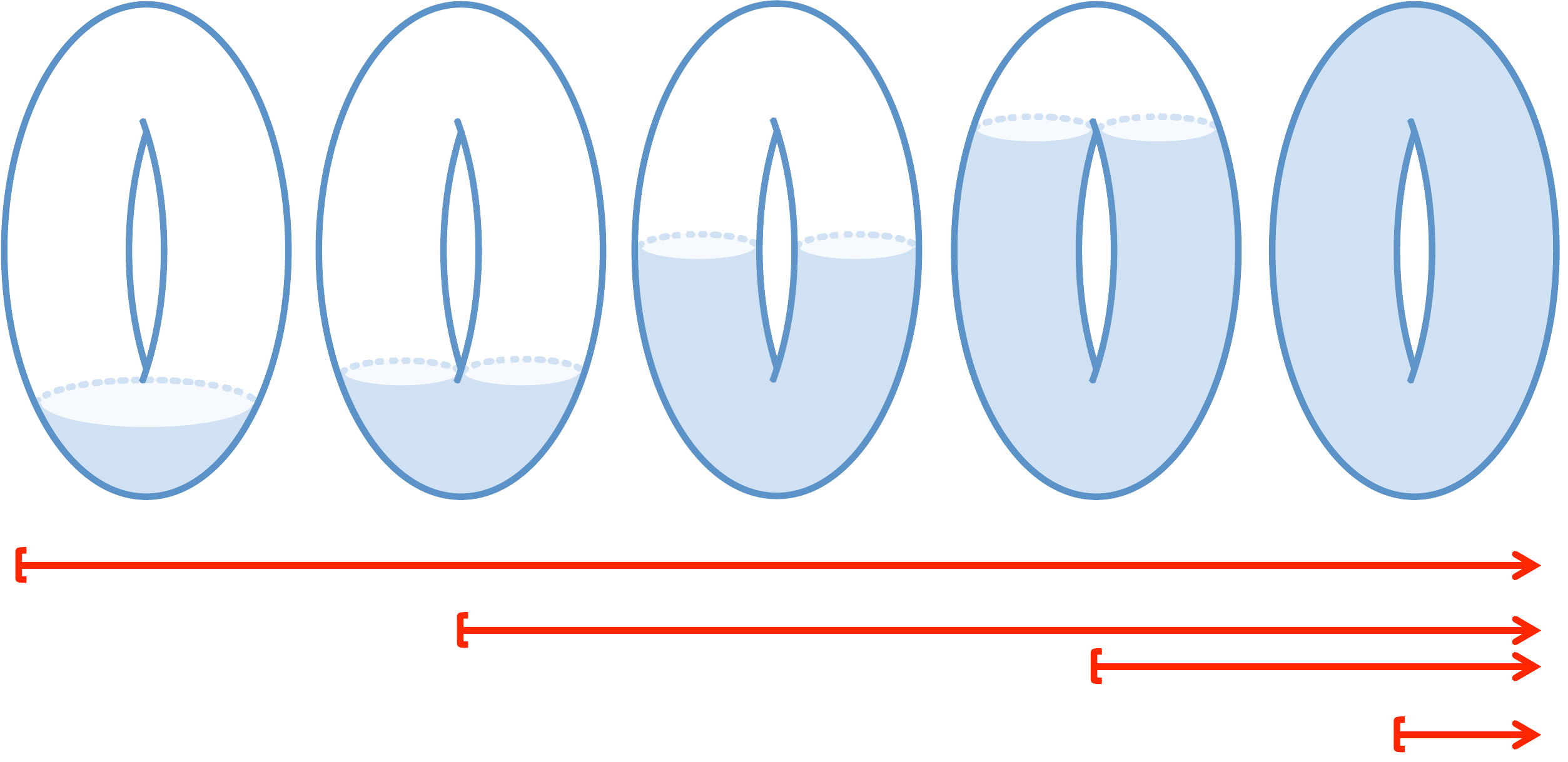}
    \caption{Barcodes for the filtration of the torus described in Example~\ref{ex:bott-torus}.}
    \label{fig:persistence_torus}
\end{figure}

\begin{ex}[Barcodes for Bott's Torus]\label{ex:bott-torus}
Consider the standard height function on the torus $h:X\to\R$, whose sub-level sets are depicted in Figure~\ref{fig:persistence_torus}. This example was first popularized by Raoul Bott~\cite{morse-theory-indomitable}. The function on the torus can be locally described in a neighborhood $U$ as a function $f|_U:\R^2\to\R$. If one calculates the matrix of partial derivatives $[\frac{\partial^2 f}{\partial x_i \partial x_j}]$ at a critical point $p\in U$ (point where $\nabla f(p)=0$), then the number of negative eigenvalues defines the \textbf{index} of the function at $p$. What Morse theory says for this example is that at each critical value the homology of the sub-level set changes by introducing homology in degree equal to the index of the corresponding critical point. The top bar in Figure~\ref{fig:persistence_torus} is the barcode for the $H_0$ persistence module, the middle two bars determine the barcode for the $H_1$ persistence module, and the final bar is the barcode for the $H_2$ persistence module.
\end{ex}

More important to applications is the freedom to choose functions other than distance for describing data.

\begin{figure}[ht]
\begin{center}
\includegraphics[width=.7\textwidth]{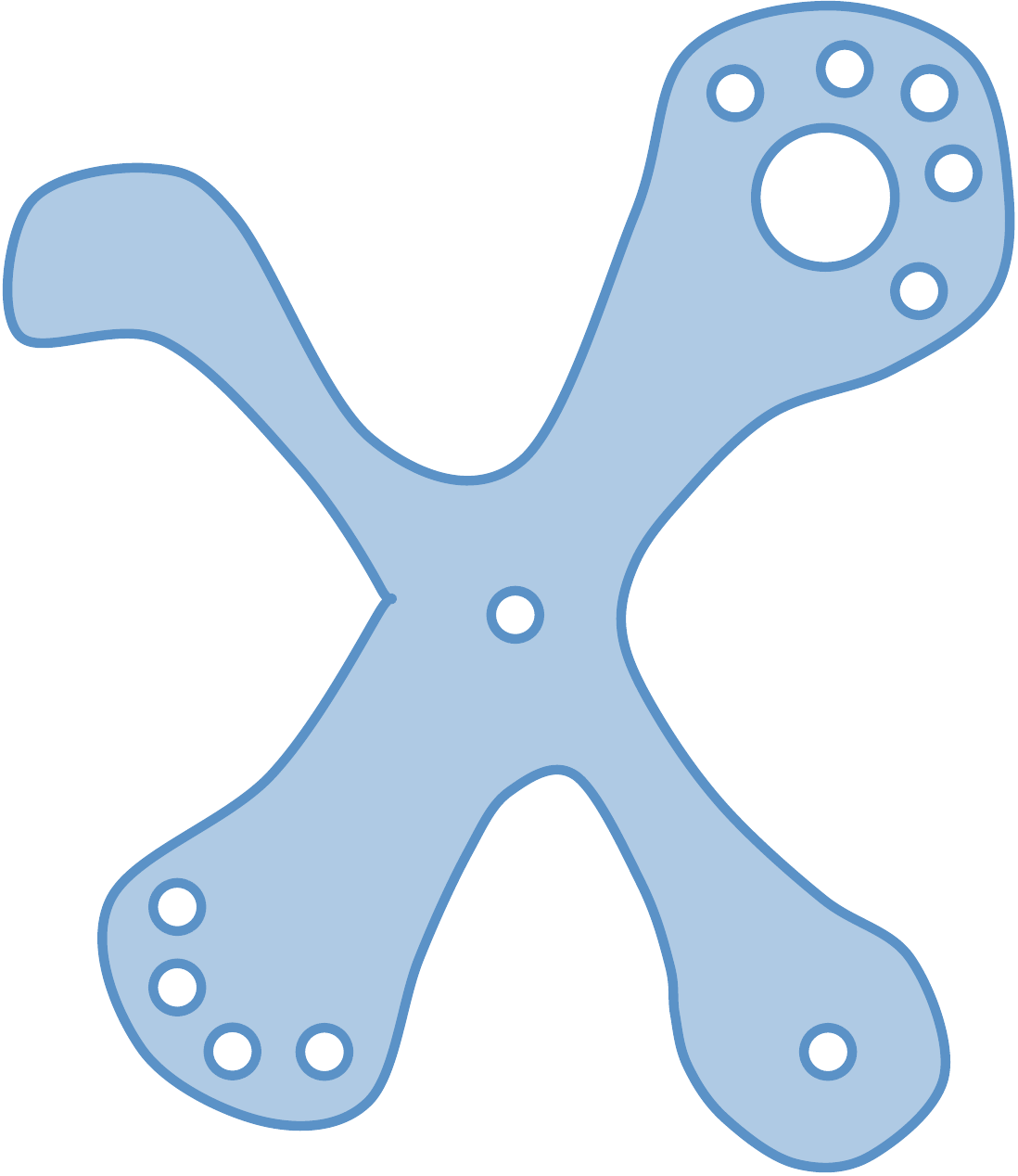}
\end{center}
\caption{A shape $X$ with four eccentric features.}
\label{fig:flare_holes}
\end{figure}

\begin{ex}[Eccentricity]

Suppose $X$ is the shape depicted in Figure \ref{fig:flare_holes}. A common feature of interest in applications~\cite{lum2013extracting} is the presence of \emph{flares} or \emph{tendrils}. Persistence provides a method for detecting such features. Consider the $p^{th}$ \textbf{eccentricity functional} on $X$:
\[
E^p(x):=\left(\int_{y\in X} d(x,y)^p dy\right)^{\frac{1}{p}}.
\]

If we filter by superlevel sets, the four endpoints of the perceived flares in Figure \ref{fig:flare_holes} will come into view. Said using homology, there are a suitable large range of values $t$ for which $E^p_{\geq t}:=\{x\in X\,|\, E^p(x)\geq t\}$ will have
\[
H_0(E^p_{\geq t}) \cong k^4.
\]
This formally expresses the four flare-like features we see in the space $X$.
\end{ex}

\begin{rmk}
When filtering by super-level sets one gets a persistence module indexed by $\R$ with its opposite total order $\leq^{op}$, so that when $s\leq t$ there is actually a map $\rho_{s,t}:H_0(E^p_{\geq t})\to H_0(E^p_{\geq s})$, but Theorem~\ref{thm:crawleyboevey} still applies.
\end{rmk}

\subsection{The Failure of Barcodes in Multi-D Persistence}

Consider again the shape in Figure~\ref{fig:flare_holes}. Suppose that we are not just interested in the number of eccentric features, but rather we are interested in holes with high eccentricity value, i.e. the persistence module
\[
H_1(E^p_{\geq t})
\]
is of interest. However, what size of hole is of interest, and what can be regarded as noise? In other words, what is the behavior of the two-parameter family of vector spaces
\[
MP_1(t,r):=H_1((E^p_{\geq t})_r)
\]
where $X_r$ denotes the set of points within distance $r$ of a subspace $X$? Extracting the general algebraic structure involved here was introduced in~\cite{carlsson2009theory}.

\begin{defn}[Multi-dimensional Persistence Module]\label{defn:multiD-module}
An \textbf{$n$-dimensional persistence module} consists of the following data: 
\begin{itemize}
\item To each point $s=(s_1,\ldots,s_n)$ in $\R^n$ a vector space $V_s$ is assigned. 
\item If $t=(t_1,\ldots,t_n)$ is another point in $\R^n$ such that $s_i\leq t_i$ for $1\leq i\leq n$ (we'll say $s\leq t$ for short), then a map of vector spaces $\rho_{t,s}:V_s\to V_t$ is assigned.
\item These maps must satisfy the property that if $r\leq s\leq t$ then $\rho_{t,r}=\rho_{t,s}\circ\rho_{s,r}$.
\end{itemize}
\end{defn}

However, as illustrated in~\cite{carlsson2009theory}, there is no higher-dimensional analog of Theorem~\ref{thm:crawleyboevey}: Not every multi-D persistence module splits as sum of constant persistence modules supported on simple pieces, like bars or their naive higher-dimensional analogs.

\section{Level Set Persistence: Towards Cosheaves}
\label{sec:level-set-persistence} 

There are many situations where the definition of a multidimensional persistence module is the correct tool for organizing data. For instance, if one has two functions of interest $f_1,f_2:X\to\R$, then taking the intersection of the sub-level sets $\{f_1(x) \leq s_1\}$ and $\{f_2(x)\leq s_2\}$ leads naturally to the 2-D persistence module
\[
  (s_1,s_2) \rightsquigarrow H_i(\{x\,|\, f_i(x)\leq s_i\,i=1,2\}).
\]
However, if one starts with a vector-valued function $f:X\to\R^2$, then it isn't clear that filtering by intersections of sub-level sets is the right method of study. In particular, if one were to post-compose the map $f:X\to\R^2$ by an isometry, one would obtain an entirely different multi-D persistence module. In short: lack of foreknowledge of the interpretations of the individual components of a vector-valued function on $X$ can severely undermine the efficacy of studying multi-D persistence.

\begin{figure}[ht]
\begin{center}
\includegraphics[width=.7\textwidth]{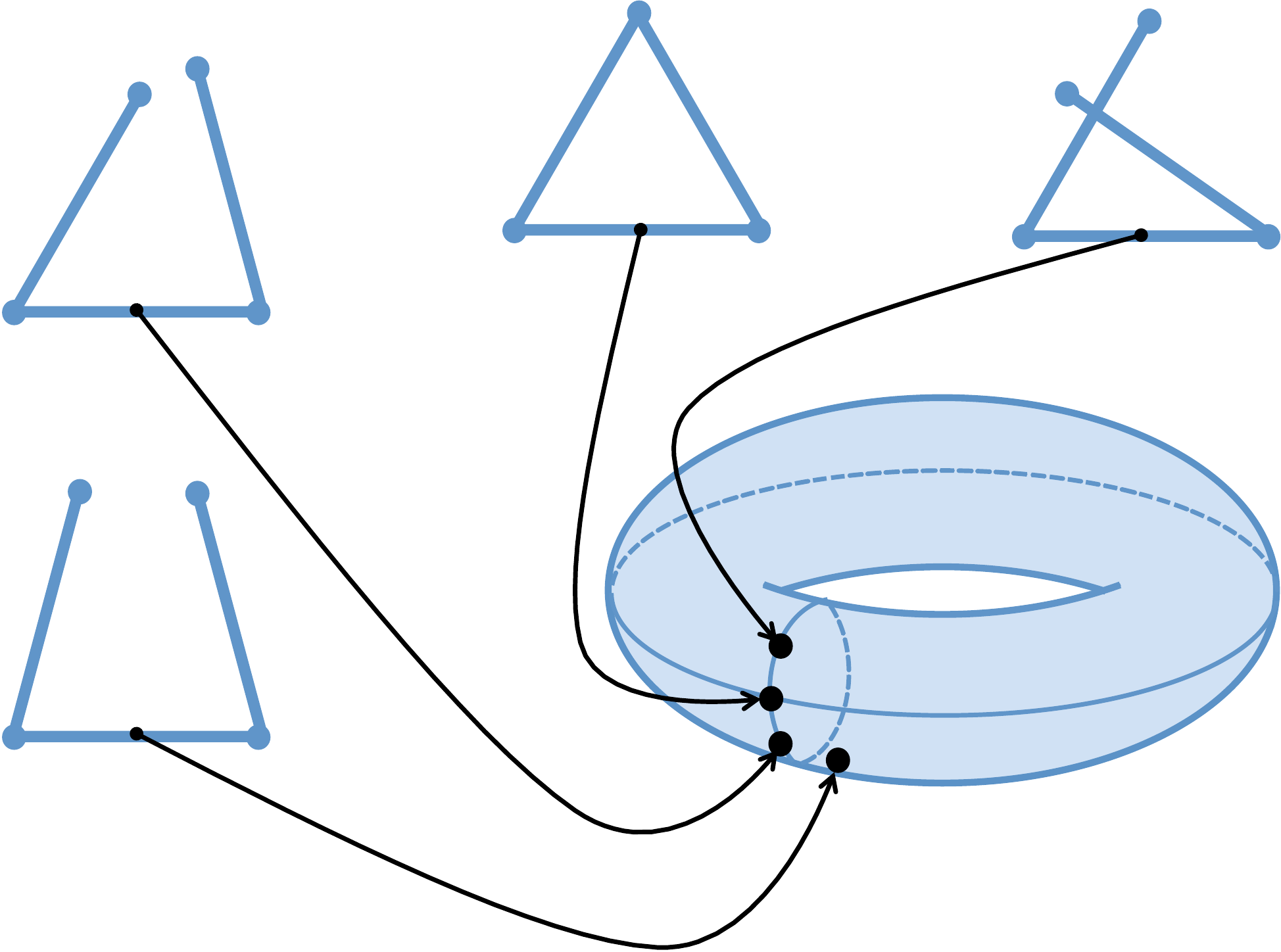}
\end{center}
\caption{A family of linkages parametrized by the torus.}
\label{fig:linkage-torus}
\end{figure}

Also, there are many situations where we want to understand how the shape of something evolves over a parameter space that is more interesting than $\R^n$, such as a space that has no natural partial order. In Figure~\ref{fig:linkage-torus} we have a linkage in the plane with two degrees of freedom corresponding to the two joints. As the angle of the two joints varies over the torus, the linkage, viewed as a subset of $\R^2$, has zero and non-zero $H_1$. How do we track the evolution of the homology as a function of the torus?

In this example, as well as several other situations that occur in data analysis~\cite{zigzag}, the natural object of study is not the homology of a sub-level set, but rather the natural object of study is the homology of the level set, or fiber, of a map $f:X\to Y$. Moreover, every sub-level set persistence problem can be cast as a level set persistence problem since we can take the sub-level sets of a map $f:X\to\R$ and construct a new space
\[
Y=\{(x,t)\in X\times\R\,|\,g(x)\leq t\}
\]
such that the fibers of the projection map $\pi:Y\to\R$ are precisely the sub-level sets of $f$. Consequently, any foundation for level-set persistence will provide a foundation for all of traditional persistence.

\subsection{Simplicial Cosheaves}\label{subsec:simplicial-cosheaves}

\begin{figure}[ht]
\begin{center}
\includegraphics[width=.7\textwidth]{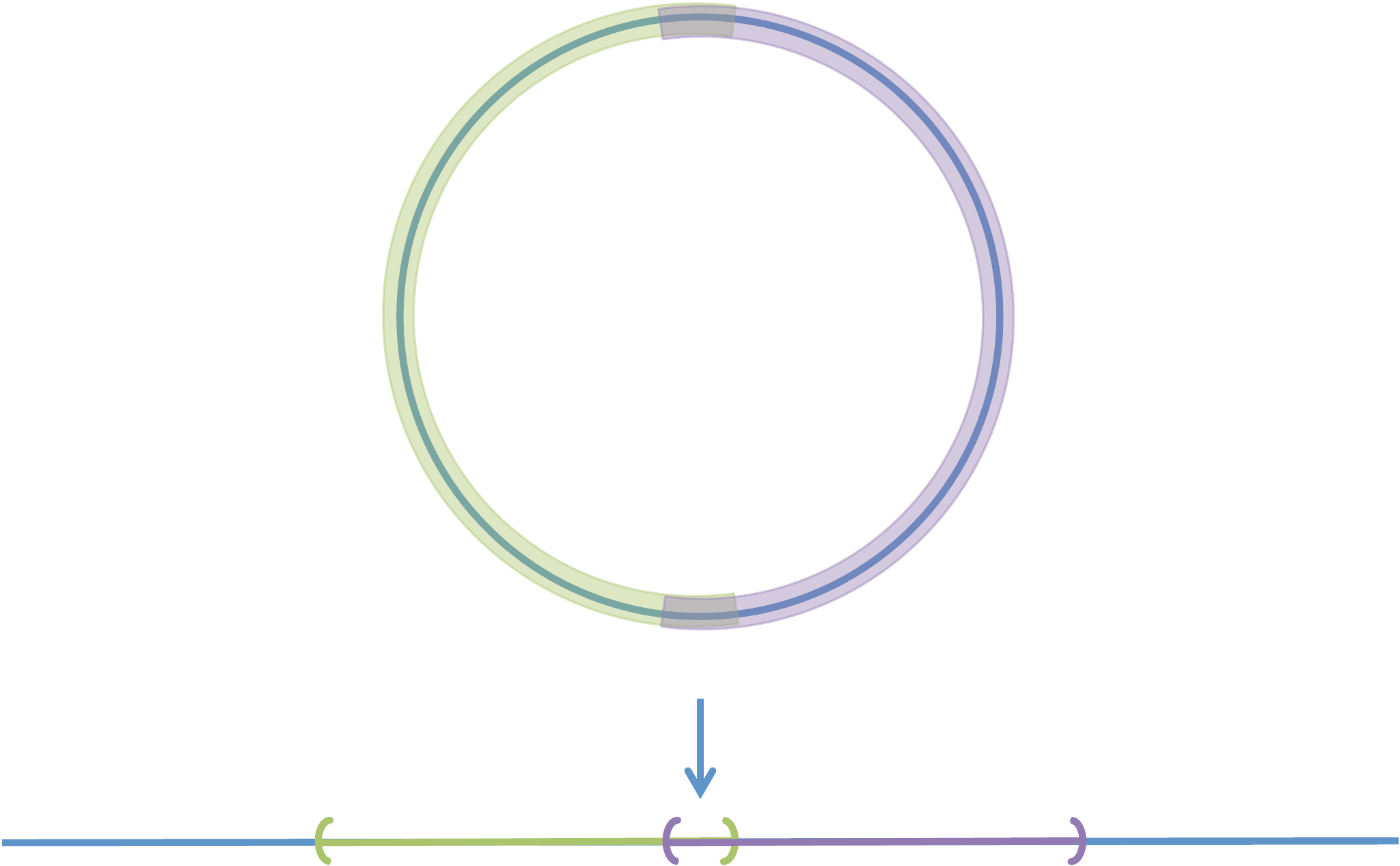}
\end{center}
\caption{A cover $\cU$ of the image of the circle $f(S^1)$ pulls back to a cover of $S^1$}
\label{fig:leray-cover-circle}
\end{figure}

The first apparent challenge of level set persistence is that one needs to relate the fibers of a map $f:X\to Y$ so that functoriality can distinguish true persistent features from spurious ones. One obvious solution is to use a cover of the image $f(X)\subseteq Y$ and then use the nerve to parametrize the homology of the pre-image. This leads to the notion of a simplicial cosheaf.

First, we make a technical observation: every simplicial complex $K$ has the structure of a partially ordered set, where one defines the partial order via inclusion of subsets of $V$, i.e.
\[
\sigma\leq\tau \Leftrightarrow \sigma\subset\tau.
\]
In the above situation one says that $\sigma$ is a \emph{face} of $\tau$.

\begin{defn}
Let $K$ be a simplicial complex. A \textbf{simplicial cosheaf over $K$} consists of an assignment of a vector space (or set) $\hF(\sigma)$ to every simplex $\sigma$ of $K$ and a map $r_{\sigma,\tau}:\hF(\tau)\to\hF(\sigma)$ for each pair of faces $\sigma\leq\tau$. The maps must satisfy $r_{\sigma,\gamma}\circ r_{\gamma,\tau}=r_{\sigma,\tau}$ whenever there is a triple of simplices $\sigma\leq\gamma\leq\tau$.
\end{defn}

\begin{ex}[Constant Cosheaf]
The assignment to every simplex in $K$ the vector space $k^n$ with identity maps between pairs of faces defines the \textbf{constant simplicial cosheaf}, named for the fact that the value of the cosheaf does not change from cell to cell.
\end{ex}

\begin{defn}[Simplicial Leray Cosheaf]
Suppose a continuous map $f:X\to Y$ is provided, as well as a cover $\cU$ of $f(X)\subseteq Y$ by open sets. For each integer $i\geq 0$ we have the \textbf{Leray simplicial cosheaf} over the nerve $N_{\cU}$ via the assignment
\[
\hF_i: \sigma \rightsquigarrow H_i(f^{-1}(U_{\sigma})).
\]
\end{defn}

\begin{ex}[Height Function on the Circle]
In Figure~\ref{fig:leray-cover-circle} we have drawn a map $f:S^1\to\R$ as well as a cover of the image. In Figure~\ref{fig:leray-circle-cosheaf} we have indicated the only Leray cosheaf of interest, where $i=0$.
\end{ex}

\begin{figure}[ht]
\begin{center}
\includegraphics[width=.7\textwidth]{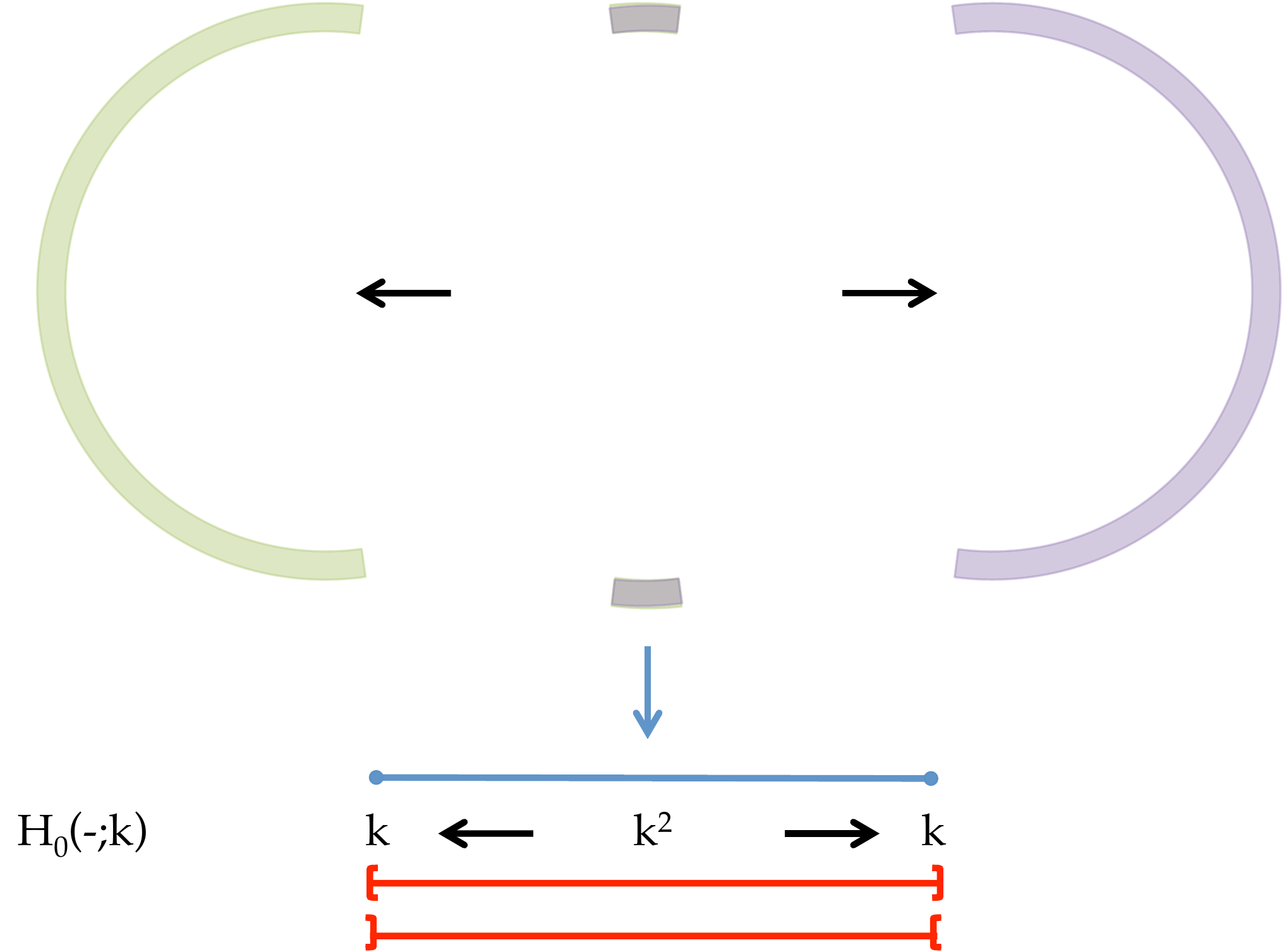}
\end{center}
\caption{The simplicial Leray cosheaf for the cover and map in Figure~\ref{fig:leray-cover-circle} and accompanying barcode, which exists by Remark~\ref{rem:undirected-bars}.}
\label{fig:leray-circle-cosheaf}
\end{figure}

\begin{rmk}[Simplicial Sheaves]\label{rmk:simplicial-sheaf}
If one uses cohomology instead of homology, then the assignment
\[
F^i: \sigma \rightsquigarrow H^i(f^{-1}(U_{\sigma}))
\]
is not a simplicial cosheaf, but rather defines a \textbf{simplicial sheaf}. The difference is small, one now has linear maps $\rho_{\tau,\sigma}:F(\sigma)\to F(\tau)$ whenever $\sigma\leq\tau$ and these satisfy the compatibility condition that whenever $\sigma\leq\gamma\leq\tau$ then $\rho_{\tau,\sigma}=\rho_{\tau,\gamma}\circ\rho_{\gamma,\sigma}$. In the constructions below, the reader may want to try dualizing a construction for simplicial cosheaves into one for simplicial sheaves.
\end{rmk}

\subsection{Homology of Barcodes via Cosheaf Homology}

One of the disturbing features of Figure~\ref{fig:leray-circle-cosheaf} is that we have no apparent way of capturing the circle's non-trivial $H_1$. This is, in fact, not true, but one needs to develop a homology theory for simplicial cosheaves in order to see why. The upshot is that \emph{data} over a simplicial complex has a homology theory and this homology can be efficiently computed~\cite{DMT_sheaves}. In the case of the simplicial Leray cosheaves associated to a map $f:X\to\R$, we can use this homology theory to gain quick computations of the true simplicial homology of the domain $X$.

Suppose we are given a simplicial complex $K$ with ordered vertices and a simplicial cosheaf $\hF$ of vector spaces over $K$. Recall that this means that to each simplex $\sigma$, we have a vector space $\hF(\sigma)$ and to each face relation $\sigma\leq\tau\in K$, we have a linear map $r_{\sigma,\tau}:\hF(\tau)\to \hF(\sigma)$. For convenience, let us adopt the following notation: if $\tau=[v_{i_0},\ldots,v_{i_p}]$, then let 
\[
\partial\tau_j=[v_{i_0},\ldots,v_{i_{j-1}},v_{i_{j+1}},\ldots,v_{i_p}]
\]
denote the $j^{th}$ face of the simplex $\tau$.

\begin{defn}
With the above notation understood, given a simplicial complex $K$ and a simplicial cosheaf $\hF$ we define the \textbf{boundary of a vector} $v\in \hF(\tau)$ by the following formula:
\[
\partial(v)=(r_{\partial\tau_0,\tau}(v),-r_{\partial\tau_1,\tau}(v),\ldots,(-1)^p r_{\partial\tau_p,\tau}(v))^{T}\in \bigoplus_{j=0}^p \hF(\partial\tau_j)
\]
\end{defn}

\begin{defn}[Simplicial Cosheaf Homology]
Given a simplicial complex $K$ and a simplicial cosheaf $\hF$, define the \textbf{group of chains valued in $\hF$} to be the direct sum of the vector spaces that $\hF$ assigns to each $p$-simplex, i.e.
\[
  C_p(K;\hF)=\bigoplus_{\tau} \hF(\tau) \qquad |\tau|=p+1.
\]
The above formula for the boundary of a vector extends to a \textbf{boundary operator}
\[
  \partial:C_{p+1}(K;\hF)\to C_p(K;\hF)
\]
that satisfies $\partial^2=0$, whence comes \textbf{simplicial cosheaf homology}:
\[
  H_p(K;\hF)= \frac{\ker \partial_p}{\im \partial_{p+1}}
\]
\end{defn}

\begin{rmk}
One can in similar fashion dualize the above constructions to define simplicial sheaf cohomology. It is unfortunate that the order of historic events has led homology to being named first and then sheaves second, because whereas sheaves have cohomology, cosheaves have homology. 
\end{rmk}

To get a handle on the above construction, let us consider cosheaf homology for the four basic simplicial cosheaves over the simplicial complex defined in Example~\ref{ex:interval}, where $K$ has three oriented simplices $[x]$, $[y]$ and $a=[x,y]$.

\begin{ex}[Closed Interval]
 Let $\hF$ be the constant cosheaf so that $\hF(x)=\hF(y)=\hF(a)=k$. The one and only boundary operator of interest is
\[
  \partial_1:\hF(a)\to \hF(x)\oplus \hF(y) \qquad \partial_1=\begin{bmatrix} 1 \\ -1\end{bmatrix}.
\]
From this we can read off the homology of $\hF$,
\[
  H_0(K;\hF)=\frac{\ker \partial_0}{\im \partial_{1}}=\frac{k^2}{k}=k \qquad H_1(K;\hF)=\frac{\ker \partial_1}{\im \partial_{2}}=\frac{0}{0}=0
\]
which agrees with the answer computed in Example~\ref{ex:interval-homology}. This agreement is obvious: simplicial cosheaf homology for the constant cosheaf $k$ is exactly the same as simplicial homology of the underlying simplicial complex.
\end{ex}

\begin{ex}[Half-Open Interval]
Consider the cosheaf $\hF$ that assigns $k$ to $x$ and $a$, but assigns $0$ to $y$. This time the boundary operator of interest is
\[
  \partial_1:k\to k \qquad \partial_1=\begin{bmatrix} 1 \end{bmatrix}.
\]
From this we can read off the homology of $\hF$:
\[
  H_0(K;\hF)=0 \qquad H_1(K;\hF)=0
\]
\end{ex}

\begin{ex}[Open Interval]
The cosheaf for this example assigns $0$ to $x$ and $y$, but $k$ to $a$. The boundary operator of interest is
\[
  \partial_1:k\to 0 \qquad \partial_1=0.
\]
From this we can read off the homology of $\hF$:
\[
  H_0(K;\hF)=0 \qquad H_1(K;\hF)=k.
\]
\end{ex}

The above computations are fundamental for the following reason. By Remark~\ref{rem:undirected-bars}, Theorem~\ref{thm:crawleyboevey} provides barcodes for simplicial cosheaves over $K$ as long as $K$ is \textbf{linear}, i.e.~$K$ is a graph where every vertex has degree at most two and contains no cycles. Consequently, we can phrase the above computations in terms of the barcode decomposition of a simplicial cosheaf over a linear complex:
\begin{quote}
\centering
{\em $H_0(K;\hF)$ counts \textbf{closed bars} and $H_1(K;\hF)$ counts \textbf{open bars}.}
\end{quote}

This observation is, at the moment, a mere curiosity. However when wedded with the following classical theorem it provides a powerful result in homology:

\begin{thm}
\label{thm:leraysheaf}
  Let $f:X \to Y$ be continuous. Assume a cover $\cU$ of the image $f(X) \subset Y$ whose nerve $N_\cU$ is at most one-dimensional, i.e.~the nerve has at most 1-simplices. For each $i\geq 0$, we have
\begin{equation*}
    H_{i}(X) \cong H_0(N_\cU;\hF_i) \oplus H_1(N_\cU;\hF_{i-1}).
\end{equation*}
\end{thm}
The proof of this result is outside of the scope of this paper, but can be found in many references~\cite{mccleary,DMT_sheaves,sca}.

Let us now compute the homology of the torus via two methods:
\begin{enumerate}
\item By computing directly the simplicial cosheaf homology of the Leray cosheaves.
\item By determining the barcodes for each of the cosheaves and applying the observation about closed and open bars.
\end{enumerate}

\begin{ex}[Height function on the Torus]

\begin{figure}[ht]
\centering
\includegraphics[width=.7\textwidth]{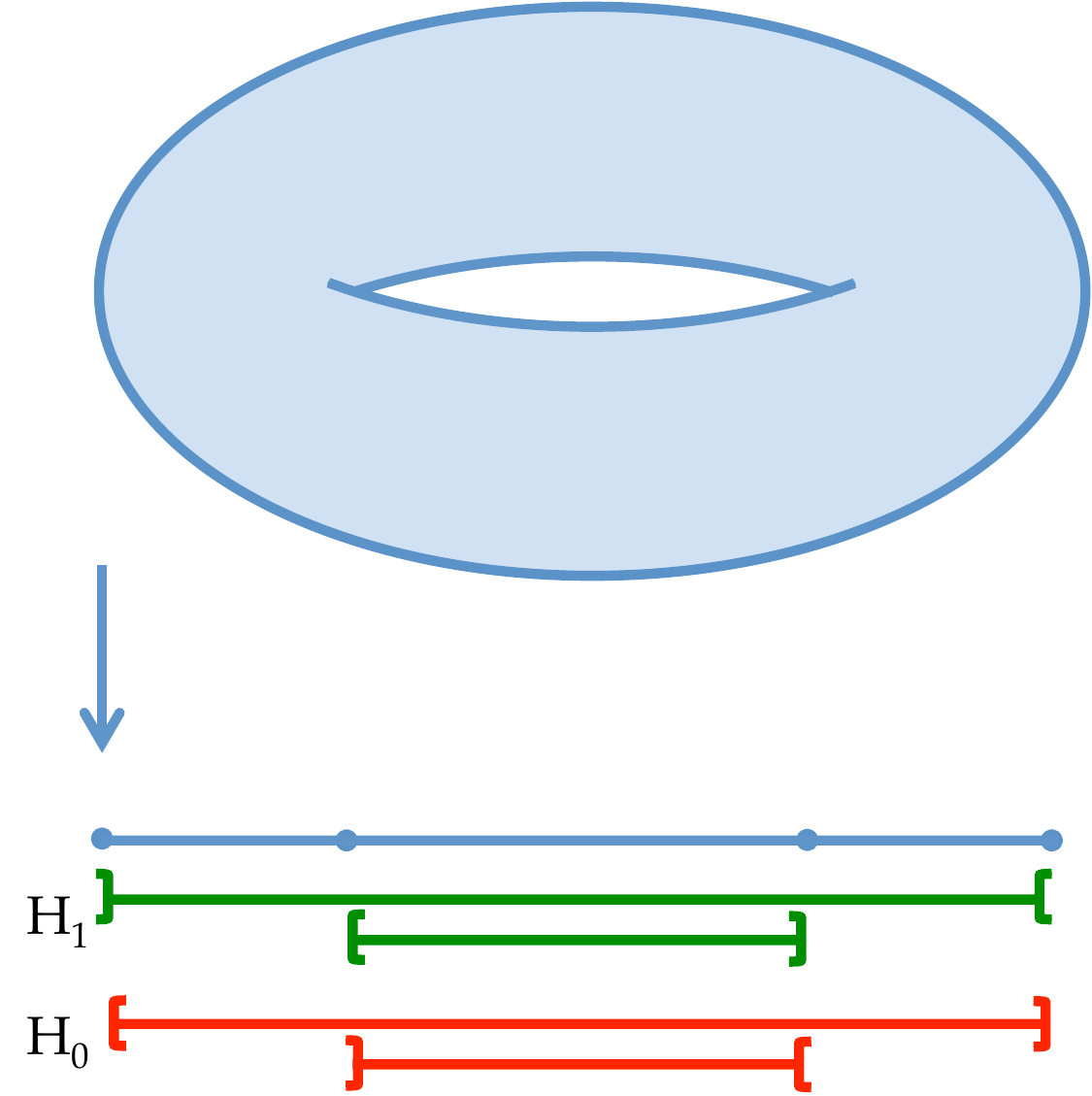}
\caption{Barcodes for Leray cosheaves coming from the height function on the torus.}
\label{fig:torus_bc}
\end{figure}

Let us now reconsider the height function on the torus $h:T\to \R$ by studying pre-images of elements of a cover. In Figure \ref{fig:torus_bc} we have omitted the cover of the image, but one can take any sufficiently large interval around each of the vertices indicated in the figure. For the sake of brevity, let us write out only the cosheaf $\hF_1$:
\[
\xymatrix{0 & \ar[l] k_a \ar[r] & k_y^2 & \ar[l] k_b^2 \ar[r] & k_z^2 & \ar[l] k_c \ar[r] & 0}
\]
Here the maps from $k_a$ to $k_y^2$ and $k_c$ to $k_z^2$ are the diagonal maps
\[
 r_{y,a}=\begin{bmatrix}1 \\ 1\end{bmatrix}=r_{z,c}
\]
and the other maps are the identity. Choosing the orientation that points to the right, we get the follow matrix representation for the boundary map:
\[
  \partial_1=\begin{bmatrix}1 & -1 & 0 & 0 \\
          1 & 0 & 1 & 0 \\
          0 & 1 & 0 & -1 \\
          0 & 0 & 1 & -1
  \end{bmatrix}
  \qquad H_1(N_{\cU};\hF_1)=<\begin{bmatrix} 1 \\ 1 \\1 \\1\end{bmatrix}> \qquad H_0(N_{\cU};\hF_1)\cong k
\]
However, if we change our bases as follows
\[
  \begin{bmatrix} y_1'= y_1 \\ y'_2=y_1+y_2\end{bmatrix} \qquad \begin{bmatrix} b_1'= b_1 \\ b'_2=b_1+b_2\end{bmatrix} \qquad \begin{bmatrix} z_1'= z_1 \\ z'_2=z_1+z_2\end{bmatrix}
\]
then our cosheaf $\hF_1$ can then be written as the direct sum of two interval modules:
\[
\xymatrix{0 & \ar[l] 0 \ar[r] & k_{y'_1} & \ar[l] k_{b'_1} \ar[r] & k_{z'_1} & \ar[l] 0 \ar[r] & 0}
\]
\[
\xymatrix{0 & \ar[l] k_a \ar[r] & k_{y'_2} & \ar[l] k_{b'_2} \ar[r] & k_{z'_2} & \ar[l] k_c \ar[r] & 0}
\]
Recalling that the latter interval module is an open bar, we can read off the homology of the torus $T$ by summing the vector spaces that lie in the same anti-diagonal slice, as described in Theorem~\ref{thm:leraysheaf}.
\[
\xymatrix{H_0(N_{\cU};\hF_1)=k \ar@{.}[rd] & H_1(N_{\cU};\hF_1)=k \\
H_0(N_{\cU};\hF_0)=k & H_1(N_{\cU};\hF_0)=k}
\]
\[
H_0(T)=k \qquad H_1(T)=k^2 \qquad H_2(T)=k
\]
\end{ex}

\subsection{Level Set Persistence Determines Sub-level Set Persistence}
\label{subsec:level-to-sublevel}

\begin{figure}
\begin{center}
\includegraphics[width=\textwidth]{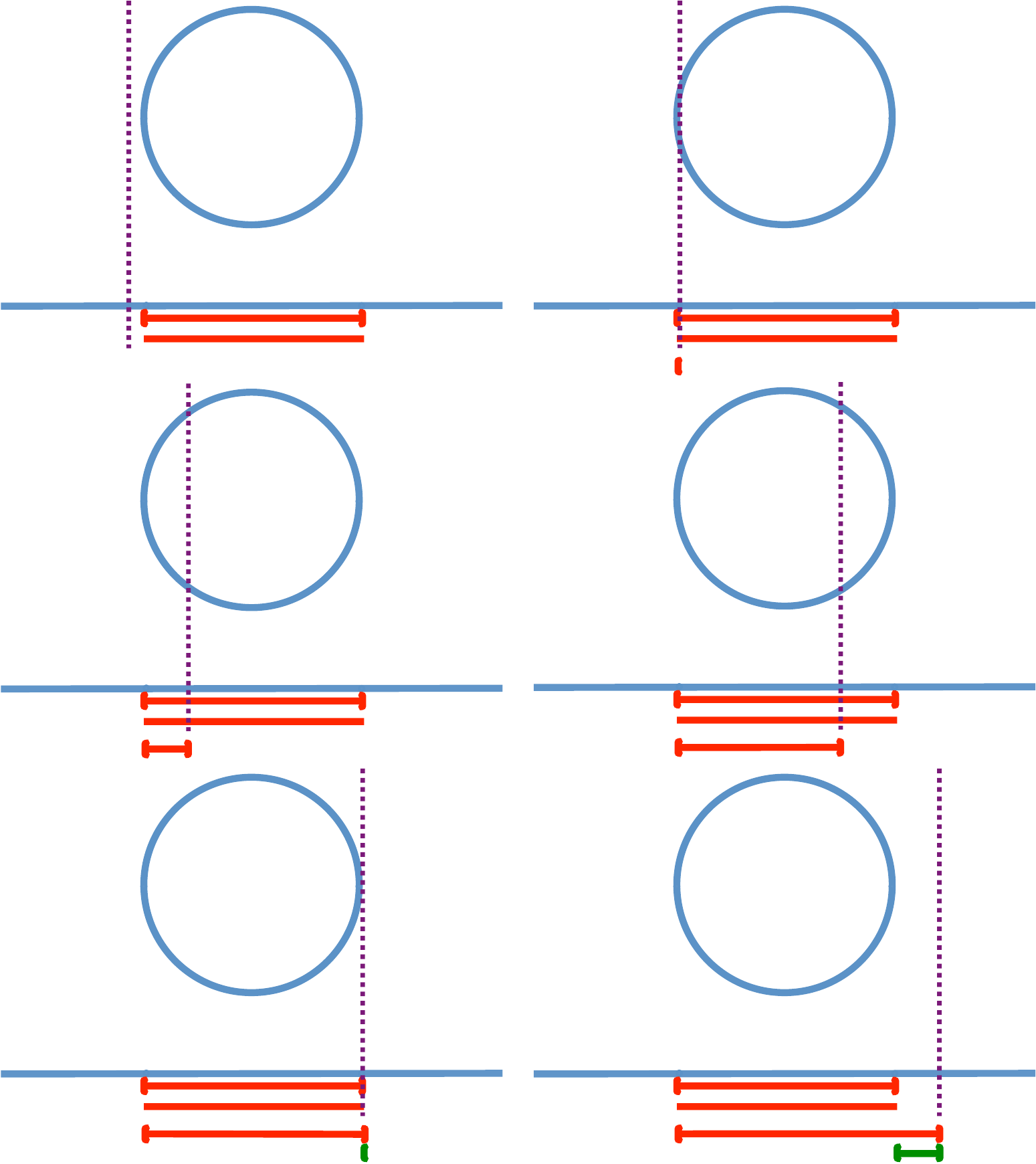}
\end{center}
\caption{Determining Sub-level Set from Level Set Persistence}
\label{fig:level_to_sub}
\end{figure}

One can also use Theorem~\ref{thm:leraysheaf} to obtain a non-obvious theorem in 1-D persistence: that level set persistence determines sub-level set persistence. By making use of the above interpretation of barcodes and cosheaf homology, we illustrate how one can take the Leray cosheaves presented as a barcode and sweep from left to right to obtain the associated sub-level set persistence module (and its barcode in certain situations). An example is drawn in Figure \ref{fig:level_to_sub}. Stated formally, we have the following theorem.

\begin{thm}
Suppose $X$ is compact and $f:X\to Y\subset \R$ is continuous. Given a cover $\cU$ of the image with linear nerve and associated simplicial Leray cosheaves $\hF_i$, one can recover the sub-level set persistence module of $f$ for any choice of $t_0< \cdots < t_n$ and integer $i\geq 0$ as follows:
\begin{enumerate}
\item For each $t_j$ take the intersection of elements in $\cU$ with the interval $(-\infty,t_j]$ to form the restricted cosheaves $\hF_i|_{(-\infty,t_j]}$ and $\hF_{i-1}|_{(-\infty,t_j]}$.
\item The persistence module in degree $i$ is then determined pointwise at $t_j$ by
\[
  H_i(f^{-1}(-\infty,t_j])\cong H_0(N_{\cU\cap(-\infty,t_j]};\hF_i)\oplus H_1(N_{\cU\cap(-\infty,t_j]};\hF_{i-1}).
\]
\end{enumerate}
\end{thm}
\begin{proof}
One must first observe that Theorem~\ref{thm:leraysheaf} holds over the restriction.
\[
  \xymatrix{f^{-1}(-\infty,t_i] \ar[r] \ar[d] & X \ar[d]^f \\
  (-\infty,t_i] \ar[r] & Y}
\]
This proves that the $i^{th}$ homology of the sub-level set can be computed via cosheaf homology. Now we must show that one can recover functoriality from the cosheaf perspective. If $\sigma\in N_{\cU}$ is a simplex in the nerve and if $t< t'$, then there is a map
\[
U_{\sigma}\cap (-\infty,t]\hookrightarrow U_{\sigma}\cap (-\infty,t'].
\]
This implies that there is a map $\hF_i(U_{\sigma}\cap (-\infty,t])\to \hF_i(U_{\sigma}\cap (-\infty,t'])$ and thus a map from chains valued in $\hF_i|_{(-\infty,t]}$ to chains valued in $\hF_i|_{(-\infty,t']}$. By functoriality of spectral sequences (maps of filtrations induce maps between spectral sequences) we get the desired map on homology.
\end{proof}

\section{Sheaves as the Correct Foundation for Level Set Persistence}

At the beginning of Section~\ref{subsec:simplicial-cosheaves}, we made a first attempt at defining level set persistence by taking a cover $\cU$ of the image of $f:X\to Y$ and studying simplicial Leray cosheaves over the nerve $N_{\cU}$. However, a problem emerges: Suppose we use a different cover $\cU'$ of the image. Is there any way of comparing the Leray simplicial cosheaves over two different nerves? Of course one could always refine the two covers $\cU$ and $\cU'$ to a common cover, but it would be convenient for proving theorems to work with all open sets at once. This leads to the general notion of a cosheaf, which is the dual notion of a sheaf. At this point we introduce a little category theory to facilitate the discussion.

\subsection{Categories and Functors}

We have used the notion of functoriality in a rather restricted way and this is how it was for much of the first part of the $20^{th}$ century. Finally in 1945, Samuel Eilenberg and Saunders Mac Lane introduced the notion of a category to make the term ``functorial'' precise and more widely applicable~\cite{em-cat45}. It has since become apparent that the language of categories provides a useful way of identifying formal similarities throughout mathematics. The success of this perspective is largely due to the fact that category theory --- as opposed to set theory --- emphasizes the relationships between objects rather than the objects themselves.

\begin{defn}[Category]\index{category}
A \textbf{category} $\cat$ consists of a class of objects $\obj(\cat)$ and a set of morphisms $\Hom_{\cat}(a,b)$ between any two objects $a,b\in\obj(\cat)$. An individual morphism $f:a\to b$ is also called an arrow since it points from $a$ to $b$. We require that the following axioms hold:
\begin{itemize}
  \item Two morphisms $f\in \Hom_{\cat}(a,b)$ and $g\in\Hom_{\cat}(b,c)$ define a third morphism $g\circ f\in \Hom_{\cat}(a,c)$, called the composition of $f$ and $g$.
  \item Composition is associative, i.e. if $h\in\Hom(c,d)$, then $(h\circ g)\circ f=h \circ (g\circ f)$.
  \item For each object $x$ there is an identity morphism $\id_x\in\Hom_{\cat}(x,x)$ that satisfies $f\circ \id_a=f$ and $\id_b\circ f=f$.
\end{itemize}
When the category $\cat$ is understood, we will sometimes write $\Hom(a,b)$ to mean $\Hom_{\cat}(a,b)$.
\end{defn}

\begin{ex}[Poset]
Any partially-ordered set $(Q,\leq)$ defines a category by letting the objects be the elements of $Q$ and by declaring each set $\Hom(x,y)$ to either have a unique morphism if $x\leq y$ or to be empty if $x\nleq y$. The transitivity axiom for partially ordered sets is expressed categorically via composition of morphisms. Associativity comes from there being a unique morphism between $x$ and $y$ when $x\leq y$. The existence of identities comes from the reflexivity axiom of a poset, namely that $x\leq x$. The anti-symmetry axiom of a poset ($x\leq y$ and $y\leq x$ implies $x=y$) is unnecessary from the categorical viewpoint and offers a natural point of generalization.
\end{ex}

\begin{ex}[Open Set Category]
 The \textbf{open set category} associated to a topological space $X$, denoted $\Open(X)$, has as objects the open sets of $X$ and a unique morphism $U\to V$ for each pair related by inclusion $U\subseteq V$.
\end{ex}

\begin{ex}
$\Vect$ is the category whose objects are vector spaces and whose morphisms are linear maps.
\end{ex}

\begin{ex}[Opposite Category]
 For any category $\cat$ there is an \textbf{opposite category} $\cat^{op}$ where all the arrows have been turned around, i.e. $\Hom_{\cat^{op}}(x,y)=\Hom_{\cat}(y,x)$.
\end{ex}

\begin{rmk}[Duality and Terminology]
Because one can always perform a general categorical construction in $\cat$ or $\cat^{op}$ every concept is really two concepts. This causes a proliferation of ideas and is sometimes referred to as the \textbf{mirror principle}. The way this affects terminology is that a construction that is dualized is named by placing a ``co'' in front of the name of the un-dualized construction. Thus there are limits and colimits, products and coproducts, equalizers and coequalizers, and many more constructions.
\end{rmk}

\begin{defn}[Functor]
  A \textbf{functor} $F:\cat\to\dat$ consists of the following data: To each object $a\in\cat$ an object $F(a)\in\dat$ is associated, i.e. $a\rightsquigarrow F(a)$. To each morphism $f:a\to b$ in $\cat$ a morphism $F(f):F(a)\to F(b)$ in $\dat$ is likewise associated. We require that the functor respect composition and preserve identity morphisms, i.e. $F(f\circ g)=F(f)\circ F(g)$ and $F(\id_a)=\id_{F(a)}$. For such a functor $F$, we say $\cat$ is the \textbf{domain} and $\dat$ is the \textbf{codomain} of $F$.
\end{defn}

\subsection{Pre-Cosheaves are Functors}


\begin{defn}[Pre-Cosheaves and Pre-Sheaves]
Any functor $$\hF:\Open(X)\to\dat$$ is called a \textbf{pre-cosheaf} valued in $\dat$. We will work exclusively with pre-cosheaves of vector spaces, so that $\dat=\Vect$. This terminology comes from dualizing a \textbf{pre-sheaf}, which is any functor $F:\Open(X)^{op}\to\dat$. Some further terminology is warranted: If $V\subset U$, then we usually write the \textbf{restriction maps} of a sheaf as $\rho_{V,U}^F:F(U)\to F(V)$ and the \textbf{extension maps} of a cosheaf as $r^{\hF}_{U,V}:\hF(V)\to\hF(U)$. Often we omit the superscript $F$ or $\hF$.
\end{defn}

\begin{rmk}
The prefix ``pre'' indicates that there is a more mature notion of a ``sheaf'' and a ``cosheaf.'' These notions are described precisely later in the paper.
\end{rmk}

\begin{defn}[Leray Pre-Cosheaf]\label{defn:leray-pre-cosheaf}
Given a continuous map $f:X\to Y$ and an integer $i\geq 0$, one has the \textbf{Leray pre-cosheaf}:
\[
\hP_i:U\subset Y \rightsquigarrow H_i(f^{-1}(U))
\]
Dually, one has the \textbf{Leray pre-sheaf}:
\[
P^i:U\subset Y \rightsquigarrow H^i(f^{-1}(U))
\]
\end{defn}

\begin{rmk}[A Contrasting Approach]
One approach to defining the level set persistence of a map $f:X\to Y$ is outlined in~\cite{bubenik2013metrics}. There one considers the collection of all subsets of $Y$ as a partially-ordered set and hence a category. There one defines the $i^{th}$ level set persistence to be the functor
\[
Z\subseteq Y \rightsquigarrow H_i(f^{-1}(Z)).
\]
This approach is closely connected with the Leray pre-cosheaves presented here except that one works only with the collection of open subsets of $Y$.
\end{rmk}

\begin{figure}[ht]
\centering
\includegraphics[width=.7\textwidth]{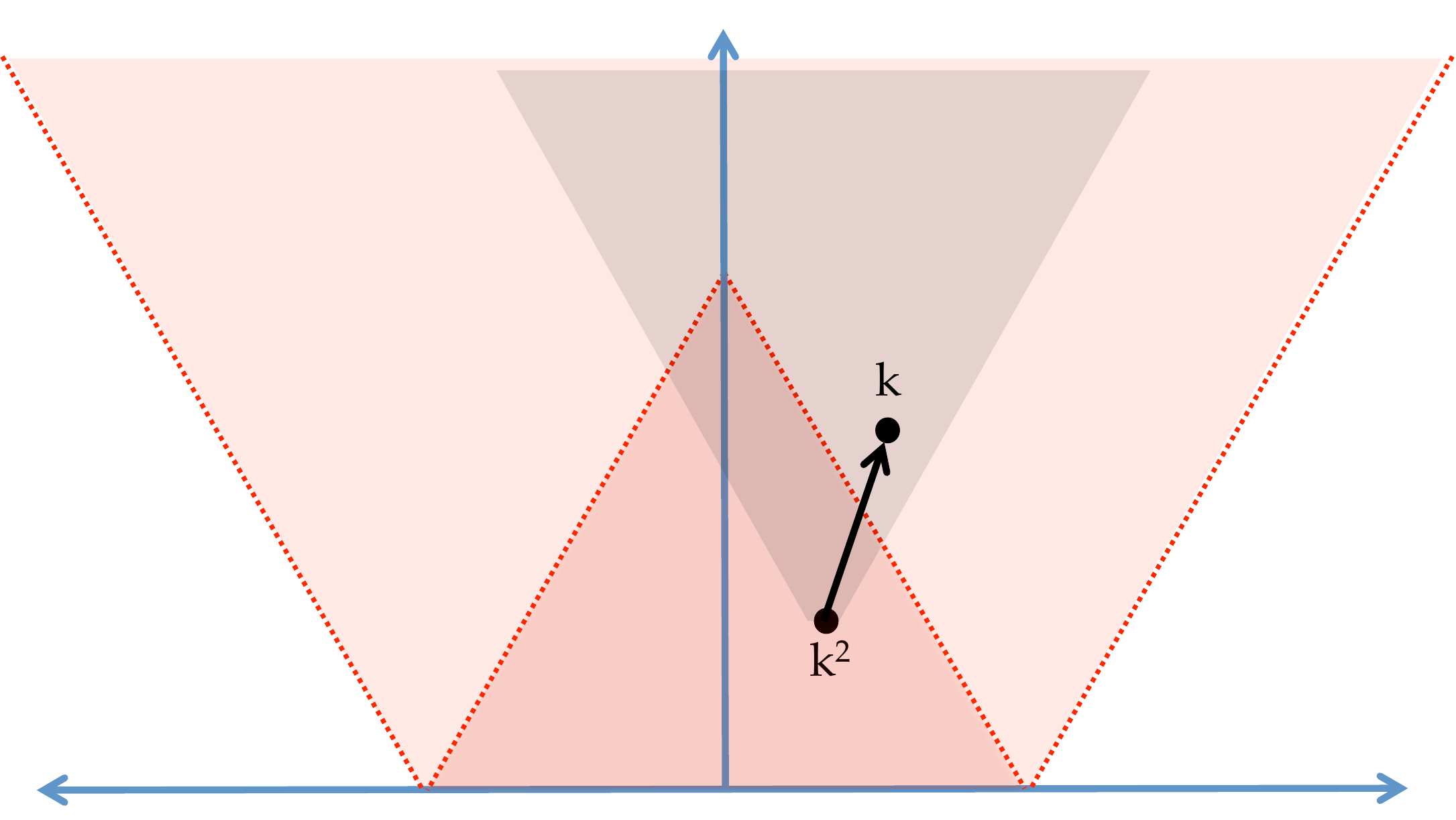}
\caption{Visualizing the Leray pre-cosheaf $H_0$ for the height function on the circle, originally considered in Figure~\ref{fig:leray-cover-circle}.}
\label{fig:light-cone-circle-h0}
\end{figure}

\begin{ex}[Height Function on the Circle]
Let $f:S^1\to\R$ be the function that projects $C=\{(x,y)\in\R^2\,|\, x^2+y^2=1\}$ onto the $x$-axis. For each open set $U$ in $\R$, $\hP_i$ assigns the $i^{th}$ homology group $H_i(f^{-1}(U))$ to $U$. Let us restrict our functor to the category of bounded open intervals $\Int(\R)$, since they generate all of $\Open(\R)$. Note that $\Int(\R)$ can be visualized as the upper half-plane $\mathbb{H}_+=\{(m,r)\,|\,m\in\R, r>0\}$ by letting each point $(m,r)$ represent the midpoint and radius of an interval $I\subset\R$:
\[
m(I)=\frac{x+y}{2} \qquad r(I)=\frac{y-x}{2}
\]
The partial order $I\leq J\Leftrightarrow I\subseteq J$ is then equivalent to the partial order on $\mathbb{H}_+$ where $(m,r)\leq (m',r')$ if and only if $|m'-m|\leq r'-r$. Thus, for maps to the real line, the Leray pre-cosheaf $\hP_i$ assigns to each point in the upper-half plane the vector space $H_i(f^{-1}(I))$, and to each pair of inclusions $I\leq J$ the map $H_i(f^{-1}(I))\to H_i(f^{-1}(J))$. For $i=0$ and the height function on the circle, this assignment is depicted in Figure~\ref{fig:light-cone-circle-h0}.
\end{ex}

\begin{rmk}
This method of visualizing the Leray pre-cosheaf is loosely inspired by the landmark paper on level set persistence~\cite{pyramid}.
\end{rmk}

\subsection{Obtaining Fibers via Stalks} 

One apparent disadvantage that Leray pre-cosheaves have is the restriction to open sets $U$ prohibits directly recording the homology of the fiber $f^{-1}(y)$. However, there is a categorical construction that can be used in some cases to derive $H_i(f^{-1}(y))$ from the homology groups $H_i(f^{-1}(U))$. Moreover, this construction will work even better when we dualize to cohomology, which motivates the use of Leray pre-sheaves.

\begin{defn}[Limit]
  The \textbf{limit} of a functor $F:\Iat\to\cat$ is an object $\varprojlim F\in\cat$ along with a collection of morphisms $\psi_x:\varprojlim F\to F(x)$ that commute with arrows in the diagram of $F$, i.e.~if $g:x\to y$ is a morphism in $\Iat$, then $\psi_y=F(g)\circ\psi_x$ in $\cat$.

  We require that the limit is universal in the following sense: if there is another object $L'$ and morphisms $\psi'_x$ that also commute with arrows in $F$, then there is a unique morphism $u:L'\to \varprojlim F$ that commutes with everything in sight, i.e. $\psi_x'=\psi_x\circ u$ for all objects $x$ in $\Iat$.
    \[
    \xymatrix{
    & L' \ar@{.>}[d]^{\exists!}_u \ar[ddl]_{\psi'_x} \ar[ddr]^{\psi'_y} & \\
    & \varprojlim F \ar[dl]^{\psi_x} \ar[dr]_{\psi_y}  & \\
    F(x) \ar[rr]_{F(g)} & & F(y) 
    }
  \]
\end{defn} 

\begin{ex}
Let $\Iat$ be the category of open sets $U$ that contain a point $y\in Y$ with morphisms corresponding to inclusions, which we call $\Open(Y)_y$. The limit of the restricted functor $\hP_i:\Open(Y)_y\to\Vect$ is called the \textbf{costalk} of $\hP_i$ at $y$. Unfortunately, for a general continuous map it is unknown how the costalk at $y$ is related to the homology of the fiber $f^{-1}(y)$. The technical reason for this is that limits and homology do not commute~\cite[Prop.~2.5.19]{sca}. This is one traditional reason why many mathematicians prefer pre-sheaves over pre-cosheaves.
\end{ex}

\begin{defn}[Colimit]
  The \textbf{colimit} of a functor $F:\Iat\to\cat$ is defined in a dual manner.
  \[
  \xymatrix{F(x) \ar[rr]^{F(g)} \ar[dr]^{\phi_x} \ar[ddr]_{\phi'_x} & & F(y) \ar[dl]_{\phi_y} \ar[ddl]^{\phi'_y} \\ & \varinjlim F  \ar@{.>}[d]^{\exists!}_u & \\ & C'  &}
  \]
\end{defn}

\begin{ex}[Stalk]
Given a pre-sheaf $F:\Open(Y)^{op}\to \Vect$ and a point $y\in Y$ the \textbf{stalk} at $y$ is defined to be the colimit of $F$ over open sets containing $y$:
\[
  F_y:=\varinjlim_{U\ni y} F(U)
\]
\end{ex}

In contrast to the Leray pre-cosheaves, the Leray pre-sheaves are traditionally considered better behaved by the following theorem.

\begin{thm}[Thm.~6.2~\cite{iversen}]
Suppose $f:X\to Y$ is a proper map between locally compact spaces. For any point $y\in Y$ we have
\[
P^i_y\cong H^i(f^{-1}(y)).
\]
\end{thm}
\begin{proof}
The bulk of the proof appears in Theorem 6.2 of~\cite[pp.~176-7]{iversen} where it is proved for the sheafification of $P^i$, which we will describe shortly. One can then observe that sheafification preserves stalks to get the desired result.
\end{proof}

\subsection{Local to Global Properties of the (Co)Sheaf Axiom}

If a topological space is equipped with a cover $\cU=\{U_i\}_{i\in I}$ and a pre-cosheaf $\hF$, then we can define a simplicial cosheaf over $N_{\cU}$ by restricting the assignment of $\hF$ to only those open sets (and their intersections) appearing in $\cU$:
\[
  \hF:\sigma \rightsquigarrow \hF(U_{\sigma})
\]
One can then compute simplicial cosheaf homology of $\hF$ on this cover, which is also called the \textbf{\v{C}ech homology of $\hF$}:
\[
  H_0(N_{\cU};\hF) \qquad H_1(N_{\cU};\hF) \qquad H_2(N_{\cU};\hF) \qquad \cdots
\]
The first term $H_0(N_{\cU};\hF)$ is used to define the cosheaf axiom, and its mirror term $H^0(N_{\cU};F)$ is used to define the sheaf axiom.

\begin{defn}\label{defn:cosheaf-axiom}
A pre-cosheaf $\hF$ of vector spaces is a \textbf{cosheaf} if for every open set $U$ and every cover $\cU$ of $U$ $$\hF(U)\cong H_0(N_{\cU};\hF).$$

Dually, a pre-sheaf $F$ of vector spaces is a \textbf{sheaf} if for every open set $U$ and every cover $\cU$ of $U$ $$F(U)\cong H^0(N_{\cU};F).$$
\end{defn}

\begin{rmk}[Local to Global]
It is often said that sheaves mediate the passage from local to global. This means that the value of $F(U)$ (the global datum) is completely determined by the values of $\{F(U_i)\}$ (the local data) where $\cU=\{U_i\}$ is a cover of $U$. This perspective has powerful implications for parallel processing; in essence, the (co)sheaf axiom is a distributed algorithm.
\end{rmk}

The first observation one can make about the cosheaf axiom is that if $U=U_1\cup U_2$ where $U_1\cap U_2=\varnothing$ and $\hF$ is a cosheaf, then $\hF(U)\cong \hF(U_1)\oplus \hF(U_2)$. Many pre-cosheaves satisfy this property without being cosheaves themselves. For example, each of the Leray pre-cosheaves $\hP_i$ satisfy this property without being cosheaves themselves.

\begin{figure}[ht]
\centering
\includegraphics[width=\textwidth]{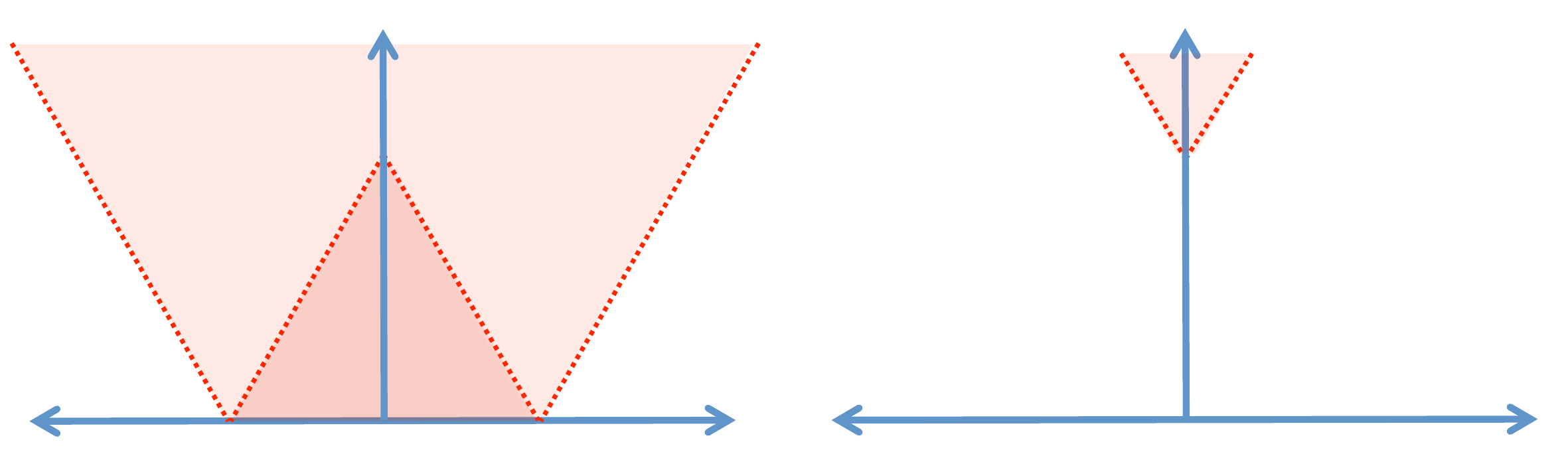}
\caption{The two Leray pre-cosheaves $\hP_0$ and $\hP_1$ for the height function on the circle. The figure on the right is an example of a pre-cosheaf that is not a cosheaf.}
\label{fig:light-cone-circle-h0h1}
\end{figure}

\begin{ex}[$\hP_1$ is not a cosheaf]
In Figure~\ref{fig:light-cone-circle-h0h1} we consider side-by-side the two non-zero Leray pre-cosheaves associated to the height function on the circle $f:S^1\to\R$. The pre-cosheaf $\hP_1$ fails to be a cosheaf because if one takes any cover $\cU=\{U_i\}$ of $f(S^1)$ by open sets where no single open set contains the entire image, then the pre-cosheaf $\hP_1$ restricts to a collection of zero vector spaces and zero maps over the nerve $N_{\cU}$. One immediately has that
\[
\hP_1(\cup U_i)=k\neq H_0(N_{\cU};\hP_1)=0,
\]
which is required in order for $\hP_1$ to be a cosheaf. On the other hand, $\hP_0$ is always a cosheaf.
\end{ex}

\begin{ex}[$\hP_0$ is a cosheaf]
Suppose $f:X\to Y$ is a continuous map. The Leray pre-cosheaf $\hP_0:U\rightsquigarrow H_0(f^{-1}(U))$ is a cosheaf. To see why, let $W=U\cup V$. By continuity of the map $f$ and the Mayer-Vietoris long-exact sequence in homology, we have the exact sequence (meaning the kernel of one map is the image of the previous) of vector spaces
  \[
 H_{0}(f^{-1}(U\cap V)) \to H_0(f^{-1}(U)) \oplus H_0(f^{-1}(V)) \to H_0(f^{-1}(W)) \to 0.
  \]
The first two terms are exactly the terms one writes down for computing \v{C}ech homology of $\hP_0$ over the cover $\{U,V\}$, i.e.
  \[
  \hP_0(U\cap V) \to \hP_0(U)\oplus\hP_0(V).
  \]
  The cokernel of this map is precisely the \v{C}ech homology of $\hP_0$ over $\{U,V\}$. The final two terms in the last row of the Mayer-Vietoris long exact sequence says precisely that $\hP_0(W)$ is isomorphic to this cokernel, i.e.
  \[
  \qquad H_0(N_{\{U,V\}};\hP_0)\cong \hP_0(W).
  \]
  Induction proves that $\hP_0$ satisfies the cosheaf condition for finite covers~\cite[p. 418]{Bredon}. To get the full cosheaf condition one then needs to use the fact that homology commutes with direct limits~\cite[p. 162]{spanier} and a technical reformulation of the cosheaf axiom~\cite[Thm.~2.3.4]{sca}.
\end{ex} 

\subsection{Sheafification and the Leray Sheaf}

Both sheaves and cosheaves have the local-to-global properties described above and so either one should be preferred over their ``pre''-cousins. Fortunately, there is a well understood procedure for turning any pre-sheaf into a sheaf called \textbf{sheafification}. It is a cruel asymmetry that there is not a similarly nice procedure for turning any pre-cosheaf into a cosheaf~\cite[Sec.~2.5.4]{sca}.

\begin{defn}[Sheafification]
Let $F:\Open(X)^{op}\to\Vect$ be a pre-sheaf. The \textbf{sheafification} $\widetilde{F}$ of $F$ assigns to every open set $U$ the set of functions $s:U\to\sqcup_{x\in U}F_x$ that locally extend, i.e.~ for every $x\in U$ and $s(x)\in F_x$ there exists a $V\ni x$ with $V\subset U$ and a $t\in F(V)$ such that the image of $t\in F(V)$ in $F_y$ agrees with $s(y)$ for all $y\in V$.
\end{defn}

\begin{defn}[Leray Sheaves]
Suppose $f:X\to Y$ is a continuous map, then the \textbf{$i^{th}$ Leray sheaf} $F^i$ is the sheafification of the Leray pre-sheaf $P^i$ associated to $f$.
\end{defn}

The assertion of this paper is that the Leray sheaves are the proper object of study for understanding the level set persistence of a proper continuous map $f:X\to Y$. Unfortunately, the Leray sheaves are uncomputable in practice and are primarily good for proving theoretical results. In principle the cosheafification of the Leray pre-cosheaves $\hP_i$ would be preferred, but there is no known cosheafification procedure.

\section{Level Set Persistence for Definable Maps}

In this section we restrict ourselves to a suitably tame class of maps and spaces so that most of the technical discrepancies between pre-sheaves and pre-cosheaves disappear. This class of maps and spaces is defined in terms of finitely many logical operations and includes most applications of interest, most notably point cloud persistence. Finally, we present the culmination of this paper: a collection of functors that can be reliably called the $i^{th}$ level set persistence of a tame map.

\subsection{Tame Topology}

\begin{defn}[\cite{vdd-ttos}, p. 2]\label{defn:o-minimal}
 An \textbf{o-minimal structure on $\RR$} is a sequence of sets $\OO=\{\OO_n\}_{n\geq 0}$ satisfying
 \begin{enumerate}
  \item $\OO_n$ is a boolean algebra of subsets of $\RR^n$, i.e. it is a collection of subsets of $\RR^n$ closed under unions and complements, with $\varnothing \in\OO_n$;
  \item If $A\in\OO_n$, then $A\times\RR$ and $\RR\times A$ are both in $\OO_{n+1}$;
  \item The sets $\{(x_1,\ldots,x_n)\in\RR^n|x_i=x_j\}$ for varying $i\leq j$ are in $\OO_n$;
  \item If $A\in\OO_{n+1}$ then $\pi(A)\in\OO_{n}$ where $\pi:\RR^{n+1}\to\RR^n$ is projection onto the first $n$ factors;
  \item For each $x\in\RR$ we require $\{x\}\in\OO_1$ and $\{(x,y)\in\RR^2|x<y\}\in\OO_2$;
  \item The only sets in $\OO_1$ are the finite unions of open intervals and points.
 \end{enumerate}
When working with a fixed o-minimal structure, we say a set is \textbf{definable} if it belongs to some $\OO_n$. A map is definable if its graph, viewed as a subset of the product, is definable.
\end{defn}

The prototypical o-minimal structure is the class of semi-algebraic sets, which has become increasingly relevant in applied mathematics.

\begin{defn}
  A \textbf{semi-algebraic} subset of $\RR^n$ is a subset of the form
  \[
    X=\bigcup_{i=1}^p\bigcap_{j=1}^q X_{ij}
  \]
  where the sets $X_{ij}$ are of the form $\{f_{ij}(x)=0\}$ or $\{f_{ij}>0\}$ with $f_{ij}$ a polynomial in $n$ variables.
\end{defn}

\begin{prop}[Semi-algebraic Sets are Definable]
The collection of semi-algebraic subsets in $\R^n$ for all $n\geq 0$ defines an o-minimal structure on $\R$.
\end{prop}
\begin{proof}
The only semi-algebraic subsets of $\RR$ are finite unions of points and open intervals. From the definition, one sees that the class of semi-algebraic sets is closed under finite unions and complements. The \textbf{Tarski-Seidenberg} theorem states that the projection onto the first $m$ factors $\RR^{m+n}\to\RR^m$  sends semi-algebraic subsets to semi-algebraic subsets~\cite{coste-sag}. We can deduce from this theorem all of the conditions of o-minimality.
\end{proof}

Semi-algebraic maps are defined to be those maps $f:\RR^k\to \RR^n$ whose graphs are semi-algebraic subsets of the product. The next example shows that the collection of augmented point clouds can be regarded as the fibers of a semi-algebraic map.


\begin{ex}[Point-Cloud Data]\label{ex:pcd}
  Suppose $Z$ is a finite set of points in $\RR^n$. For each $z\in Z$, consider the square of the distance function
  \[d_z(x_1,\ldots,x_n)=\sum_{i=1}^n (x_i-z_i)^2.\]
  By the previously stated facts we know that the sets
  \[
    B_z:=\{x\in\RR^{n+1} \, | \, d_z(x_1,\ldots,x_n)\leq x^2_{n+1}\}
  \]
  are semi-algebraic along with their unions and intersections. Denote by $X$ the union of the $B_z$. The Tarski-Seidenberg theorem implies that the map
  \[
    f:X\to\RR \qquad f^{-1}(r):=\cup_{z\in Z} B(z,r)=\{x\in\RR^n \, | \, \exists z\in Z \, \mathrm{s.t.} \, d_z(x)\leq r^2\} \]
  is semi-algebraic.
\end{ex}

One of the nice features of a point cloud is that the topology of the union $X_r=\cup_{x_i\in Z} B(x_i,r)$ only changes for finitely many values of $r$. This behavior is common among all definable sets and maps.

\begin{defn}
A definable map $f:E\to B$ between definable sets is said to be \textbf{definably trivial} if there is a definable set $F$ and a definable homeomorphism $h:E\to B\times F$ such that the diagram
\[
  \xymatrix{E \ar[rd]_{f} \ar[rr]^h & & B\times F \ar[ld]^{\pi} \\ 
  & B &}
\]
commutes, i.e.~$\pi\circ h=f$.
\end{defn} 

\begin{rmk}\label{rmk:reg-nbhd}
A definably trivial map is simple because the topology of the fiber $f^{-1}(b)\cong F$ does not change. In particular, there is a neighborhood $U$ of $b$ for which $H_i(f^{-1}(U))\cong H_i(f^{-1}(b))$, so that the costalk of the Leray pre-cosheaf agrees with the homology of the fiber. In short, there is no advantage to studying the Leray pre-sheaves over the Leray pre-cosheaves for definably trivial maps.
\end{rmk}

\begin{thm}[Trivialization Theorem~\cite{vdd-ttos}]\label{thm:tame-triv}
Let $f:E\to B$ be a definable continuous map between definable sets $E$ and $B$. Then $B$ can be partitioned into definable sets $B_1,\ldots, B_k$ so that the restrictions
\[
  f|_{f^{-1}(B_i)}:f^{-1}(B_i) \to B_i
\]
are definably trivial.
\end{thm}

\begin{figure}[ht]
  \centering
  \includegraphics[width=.5\textwidth]{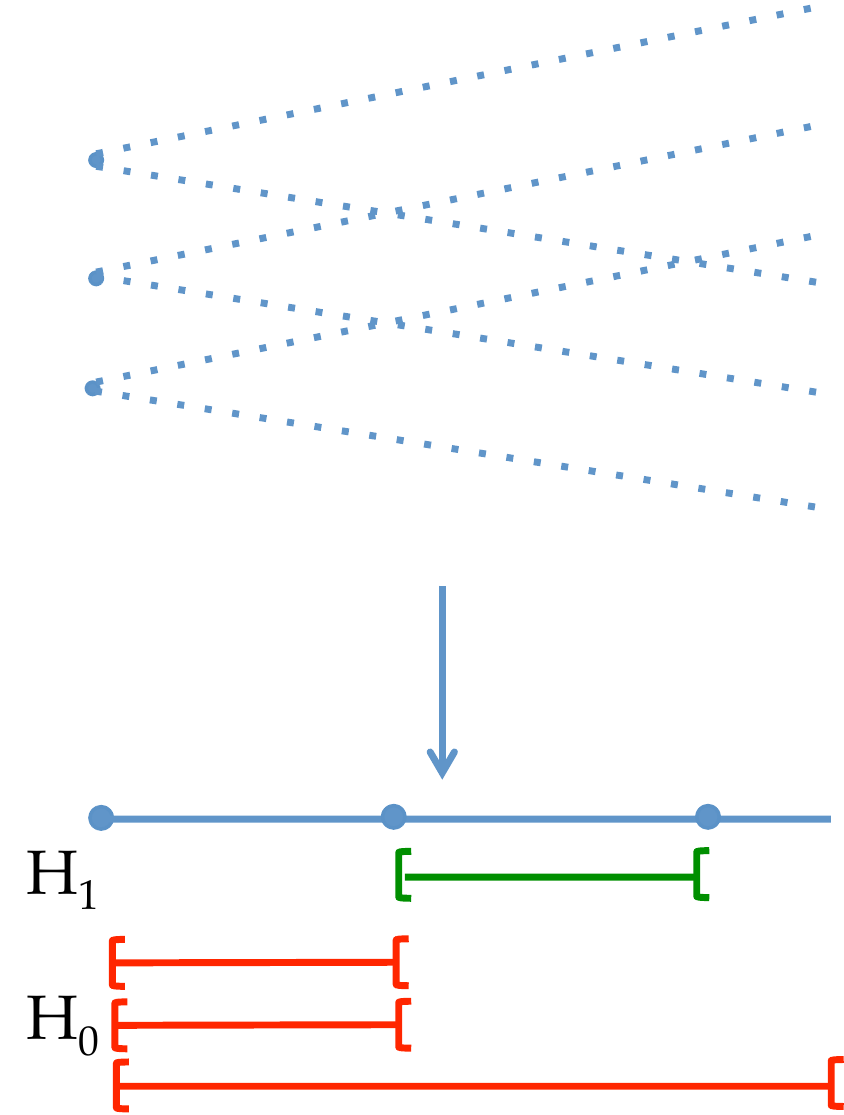}
  \caption{A point cloud consisting of three points in the plane on the edges of an equilateral triangle can be regarded as a definable map. Example~\ref{ex:point-revisited} explains how the persistence modules are constructed.}
  \label{fig:o-minimal-pointcloud}
\end{figure}

\begin{ex}[Point Cloud Revisited]\label{ex:point-revisited}
In Example~\ref{ex:pcd}, we showed that the family of augmented spaces associated to a point cloud is a definable map. This example is crucial because it shows that point cloud persistence is a special case of level set persistence. By Theorem~\ref{thm:tame-triv}, there is a decomposition of $\R$ into definable sets over which the map $f$ is definably trivial. With some work, one can show that this decomposition is into half-open intervals $\{[s_i,s_{i+1})\}$. Let $t_i$ denote a point strictly between $s_i$ and $s_{i+1}$. Letting $X_{r}=f^{-1}(r)$, one can show that there is a sequence of fibers and maps
\[
\cdots \leftarrow X_{t_i} \rightarrow X_{s_{i+1}} \leftarrow X_{t_{i+1}} \rightarrow X_{s_{i+2}} \leftarrow \cdots
\]
where every map $X_{s_i}\leftarrow X_{t_i}$ is a homeomorphism and thus an isomorphism on homology. The fact there is such an isomorphism follows from Remark~\ref{rmk:reg-nbhd}. The fact that there is a map $X_{t_i}\to X_{s_{i+1}}$ follows from the existence of a neighborhood $U$ containing $X_{t_i}$ and $X_{s_{i+1}}$ that deformation retracts onto $X_{s_{i+1}}$~\cite[Prop.~11.1.26]{sca}. Taking homology in each degree produces the persistence modules depicted in Figure~\ref{fig:o-minimal-pointcloud}.
\end{ex}

\subsection{Stratified Spaces and Constructible Cosheaves}

In this section we present the Leray (co)sheaves associated to a definable map in an entirely different way. This characterization is based on a folk-theorem of MacPherson~\cite{treumann-stacks} and is phrased in the language of Whitney stratified spaces, which includes definable sets as a special case~\cite{loi-verdier}, and constructible cosheaves, which we define below. 


\begin{defn}[Whitney Stratified Spaces]
  A \textbf{Whitney stratified space} is a space $X$ that is a closed subset of a smooth manifold $M$ along with a decomposition into pieces $\{X_{\sigma}\}_{\sigma\in P_X}$ such that
  \begin{itemize}
    \item each piece $X_{\sigma}$ is a locally closed smooth submanifold of $M$, and
    \item whenever $X_{\sigma}$ is in the closure of $X_{\tau}$ the pair satisfies \textbf{condition (b)}. This condition says if $\{y_i\}$ is a sequence in $X_{\tau}$ and $\{x_i\}$ is a sequence in $X_{\sigma}$ converging to $p\in X_{\sigma}$ and the tangent spaces $T_{y_i}X_{\tau}$ converges to some plane $T$ at $p$, and the secant lines $\ell_i$ connecting $x_i$ and $y_i$ converge to some line $\ell$ at $p$, then $\ell\subseteq T$. See Figure~\ref{fig:whitney_b}.
  \end{itemize}
\end{defn}
\begin{rmk}
  We have omitted \textbf{condition (a)} because it is implied by condition (b)~\cite[Prop. 2.4]{mather}. Condition (a) states that if we only consider a sequence $y_i$ in $X_{\tau}$ converging to $p$ such that the tangent planes $T_{y_i}X_{\tau}$ converge to some plane $T$, then the tangent plane to $p$ in $X_{\sigma}$ must be contained inside $T$.
\end{rmk}

\begin{figure}
  \centering
  \includegraphics[width=.6\textwidth]{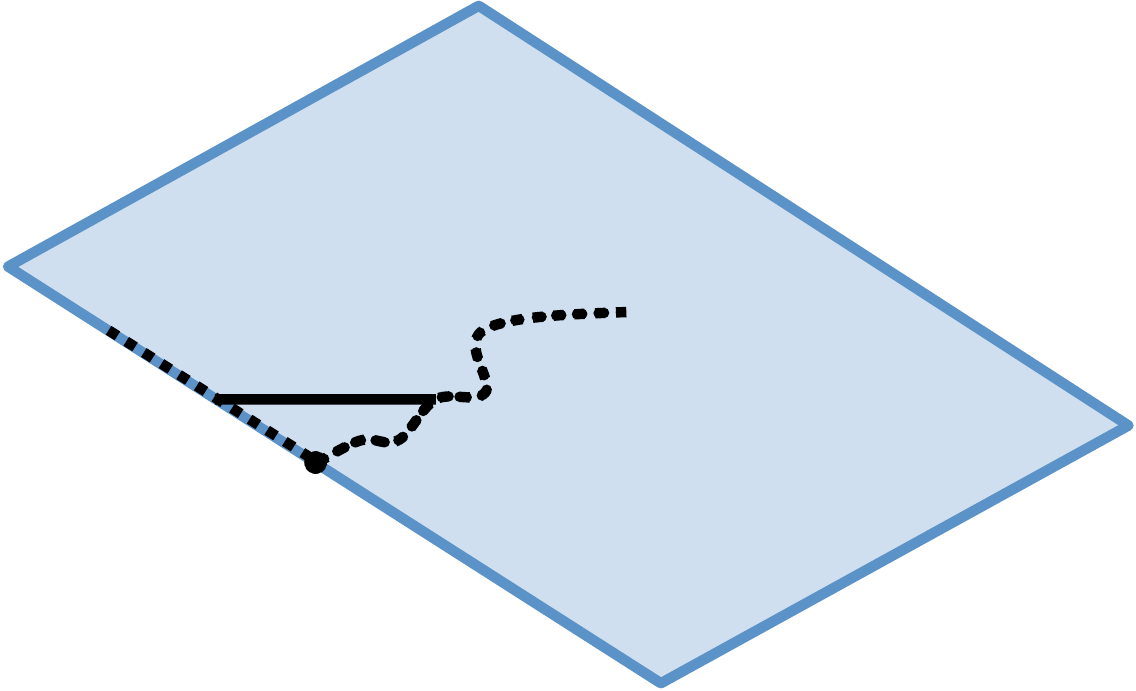}
  \caption{Diagram for Whitney Condition (b)}
  \label{fig:whitney_b}
\end{figure}

The Whitney conditions are important because so many types of spaces admit Whitney stratifications, the most important being semi-algebraic and sub-analytic spaces. Remarkably, these conditions about limits of tangent spaces and secant lines imply strong structural properties of the space, such as being triangulable~\cite{goresky-fol}.

\begin{defn}[Entrance Path Category]
Suppose $X$ is a stratified space. The \textbf{entrance path category} of $X$ $\Entr(X)$ has points of $X$ for objects and equivalence classes of entrance paths for morphisms. An entrance path is a continuous map $\gamma:I=[0,1]\to X$ with the property that the ambient dimension of the stratum containing $\gamma(t)$ is non-increasing with $t$. Two entrance paths $\gamma$ and $\eta$ connecting $x$ to $x'$ are equivalent if there is a map $h:[0,1]^2\to X$ such that for every $s\in [0,1]$ the map $h(s,t)$ is an entrance path, $\gamma(t)=h(0,t)$ and $\eta(t)=h(1,t)$; see Figure~\ref{fig:entr_path}. The \textbf{definable entrance path category} is similar with the added stipulation that $X$ is definable and that all the paths and relations are definable in the sense of Definition~\ref{defn:o-minimal}.
\end{defn}

\begin{ex}\label{ex:entrpath-simplcplx}
If $X$ is the geometric realization of a simplicial complex, then it can be stratified by its open simplices. One can prove that $\Entr(X)$ is equivalent to a poset with the relation that there is a unique entrance path from $\tau$ to $\sigma$ if and only if $\sigma\leq \tau$. We express this succinctly as
\[
  \Entr(X)\simeq (X,\leq)^{op}
\]
\end{ex}

The folk-theorem of MacPherson is that suitably behaved cosheaves defined on stratified spaces are equivalent to functors from the entrance path category. This equivalence would take us beyond the scope of this paper (see~\cite{sca} for a more thorough treatment), so we will simply define these well-behaved cosheaves as functors from the entrance path category.

\begin{figure}
  \centering
  \includegraphics[width=.5\textwidth]{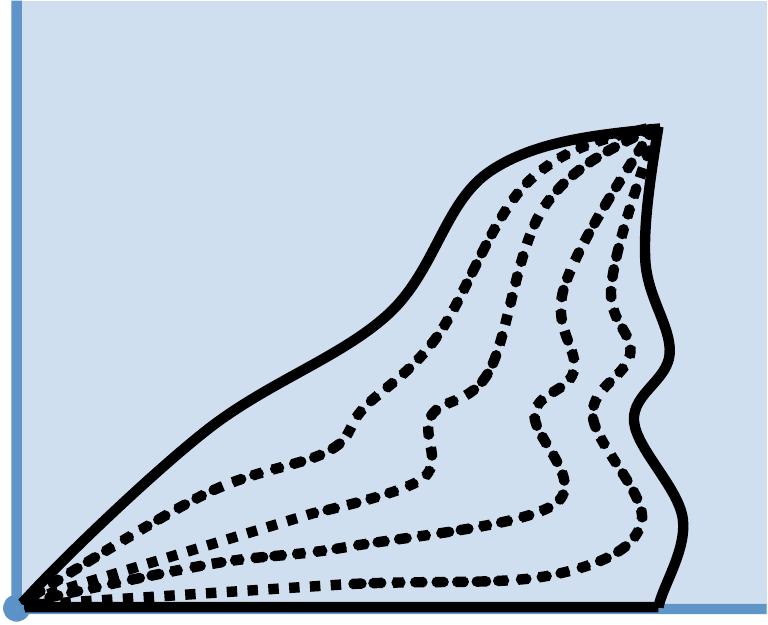}
  \caption{Two entrance paths in the plane related through a family of entrance paths.}
  \label{fig:entr_path}
\end{figure}

\begin{defn}
Suppose $X$ is a stratified space. A \textbf{constructible cosheaf} is a functor $\hF:\Entr(X)\to\Vect$.
\end{defn}

\begin{ex}
By Example~\ref{ex:entrpath-simplcplx}, we see that a simplicial cosheaf on $K$ is the same as a constructible cosheaf on the geometric realization of $K$, regarded as a stratified space.
\end{ex}

The correspondence between constructible cosheaves and actual cosheaves is encapsulated in the following theorem.

\begin{thm}[Correspondence with Cosheaves]\label{thm:reps-to-cosheaves}
Given a constructible cosheaf $\hF$ on a stratified space $X$ one can associate an actual cosheaf, which we also call $\hF$, by observing that each open set $U$ receives an induced stratification from $X$, and hence has an entrance path category, and letting
  \[
  \hF(U):= \varinjlim_{\Entr(U)} \hF|_U
  \]
\end{thm}
\begin{proof}
This is theorem 11.2.15 of~\cite{sca}. It requires proving a Van Kampen theorem for the entrance path category, which is beyond the scope of this paper.
\end{proof}

\begin{ex}
 In Figure~\ref{fig:colimit-entrance-path} we have two constructible cosheaves over the real line. For each constructible cosheaf we have picked the three open intervals and the corresponding colimit of the cosheaf over the entrance path category restricted to that open set.
\end{ex}

\begin{figure}[ht]
  \centering
  \includegraphics[width=\textwidth]{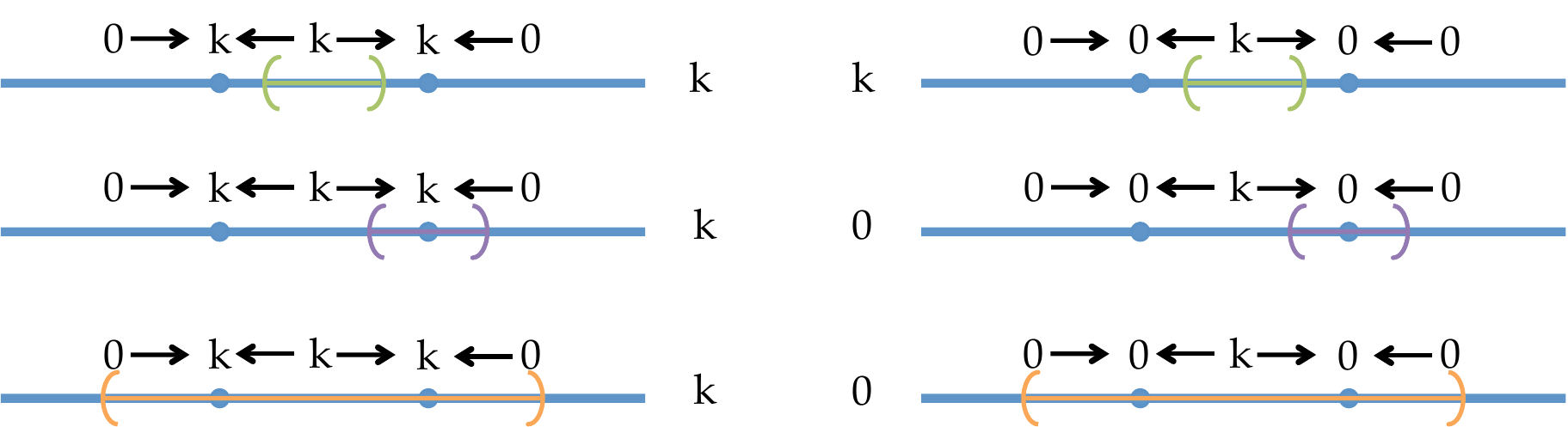}
  \caption{Two constructible cosheaves and the associated colimits of the restriction to various intervals.}
  \label{fig:colimit-entrance-path}
\end{figure}

Now we can state a definable analog of the Leray sheaves that could be programmed on a computer.

\begin{thm}[Constructible Cosheaves from Definable Maps~\cite{sca}]\label{thm:strat_maps}

  If we are given a proper definable map $f:E\to B$ that comes from the restriction of a $C^1$ map between manifolds, then for each $i$ the assignment
  \[
    b\in B \rightsquigarrow H_i(f^{-1}(b))
  \]
defines a definable cosheaf.
\end{thm}

\begin{rmk}[Sketch of the Proof]
This is a non-trivial theorem, which is proved in detail as Theorem 11.2.17 of~\cite{sca}. The first observation to make is that the fiber $f^{-1}(b)$ over a point $b\in B$ has an open neighborhood $U$ that retracts onto the fiber. This is because $f^{-1}(b)$ can be presented as a closed union of finitely many strata~\cite[p.~60]{vdd-ttos} and the closed union of finitely many strata has a regular neighborhood that retracts onto it~\cite[Prop.~11.1.26]{sca}.

Intuitively, if a path $\gamma:I\to B$ starts in a stratum $B_{\tau}$ that contains $b=\gamma(1)$ in it's closure, then one can assign the homology of the zig-zag of inclusions
\[
f^{-1}(\gamma(0))\hookrightarrow U \hookleftarrow f^{-1}(\gamma(1))
\]
to any morphism $\gamma$ in $\Entr(B)$. However, we prefer a more inductive procedure by considering the pullback $I\times_B E=\{(t,e)\,|\,\gamma(t)=f(e)\}$ as a definable set~\cite[Lem.~11.1.15]{sca} and the projection $\pi_1:I\times_B E\to[0,1]$ as a definable map.

To prove invariance under homotopy through entrance paths, one then considers a definable homotopy $h:I^2\to B$ and pulls back to a definable map to the square $I^2$. One then proves invariance for this restricted map.
\end{rmk}

\bibliographystyle{alpha}
\bibliography{REFS-revised.bib}

\end{document}